\documentclass{article}

\usepackage{Other/arxiv_reqs}
\usepackage{Other/our_macros}     % Our macros

\usepackage[utf8]{inputenc} % allow utf-8 input
\usepackage[T1]{fontenc}    % use 8-bit T1 fonts
\usepackage{hyperref}       % hyperlinks
\usepackage{url}            % simple URL typesetting
\usepackage{booktabs}       % professional-quality tables
\usepackage{amsfonts}       % blackboard math symbols
\usepackage{nicefrac}       % compact symbols for 1/2, etc.
\usepackage{microtype}      % microtypography
\usepackage{xcolor}         % colors

\title{The Burer-Monteiro SDP method can fail even above the Barvinok-Pataki bound}
\author{
    Liam O'Carroll\thanks{Northwestern University. 
    Email: \email{liamocarroll2023@u.northwestern.edu}. Supported by an NSF REU supplement associated with CCF-1934931 and an undergraduate research grant from Northwestern University.}
    \and
    Vaidehi Srinivas\thanks{Northwestern University. 
    Email: \email{vaidehi@u.northwestern.edu}.}
    \and
    Aravindan Vijayaraghavan\thanks{Northwestern University. 
    Email: \email{aravindv@northwestern.edu}. The last two authors were also supported by the National Science Foundation (NSF) under Grant Nos. CCF-1652491, CCF-1934931 and EECS-2216970.}
}
\date{}

\begin{document}

\maketitle

\begin{abstract}
	The most widely used technique for solving large-scale semidefinite programs (SDPs) in practice is the non-convex Burer-Monteiro method, which explicitly maintains a low-rank SDP solution for memory efficiency. There has been much recent interest in obtaining a better theoretical understanding of the Burer-Monteiro method. When the maximum allowed rank $p$ of the SDP solution is above the Barvinok-Pataki bound (where a globally optimal solution of rank at most \(p\) is guaranteed to exist), a recent line of work established convergence to a global optimum for generic or smoothed instances of the problem. However, it was open whether there even exists an instance in this regime where the Burer-Monteiro method fails. We prove that the Burer-Monteiro method can fail for the Max-Cut SDP on $n$ vertices when the rank is above the Barvinok-Pataki bound ($p \ge \sqrt{2n}$). We provide a family of instances that have spurious local minima even when the rank $p = n/2$. Combined with existing guarantees, this settles the question of the existence of spurious local minima for the Max-Cut formulation in all ranges of the rank and justifies the use of beyond worst-case paradigms like smoothed analysis to obtain guarantees for the Burer-Monteiro method.
\end{abstract}

%%%%%%%%%%%%%%%%%%%%%%%%%%%%%%%%%%%%%%%%%%%%%%%%%%%%%%%%%%%%
%%% UNCOMMENT IF YOU WANT TO SEE A TOC FOR EASE OF REFERENCE:
% \tableofcontents
%%%%%%%%%%%%%%%%%%%%%%%%%%%%%%%%%%%%%%%%%%%%%%%%%%%%%%%%%%%%

\section{Introduction}
\label{sec:introduction}
% Structure:

% Para 1: Introduction to SDPs
% 1. SDPs being super useful, powerful algorithmic tools.  
% 2. Several theoretical works giving polynomial time algorithms. 
% 3. However memory is a major bottleneck for realizing efficient algorithms. 

% Para 2: 
% 1. Introduce Burer-Monteiro method. 
% 2. Non-convex
% 3. Barvinok-Pataki bound. 
% Ask a more general question?

% Para 3: What is known?
% Quote their open question. 

% Para 4: How we resolve the open question with theorem.

% Para 5: Clarifying the picture. 
% Explain that our result shows that beyond worst-case analysis paradigms like smoothed-analysis is actually necessary to obtain polynomial time guarantees. 

Semidefinite programs (SDPs) are a powerful algorithmic tool with wide-ranging applications in combinatorial optimization, control theory, machine learning and operations research. Notably, they yield optimal approximation algorithms for NP-hard problems like Max-Cut~\cite{GW95} and other constraint satisfaction problems~\cite{Ragh}. While interior-point algorithms and the ellipsoid method give polynomial time guarantees for solving semidefinite programs, memory becomes a bottleneck for relatively modest instance sizes. This has prompted research into scalable semidefinite programming algorithms---see \cite{MHA20} for a recent survey.

One of the most popular methods for solving SDPs in practice is the one pioneered by Burer and Monteiro \cite{BM03,BM05}, which explicitly constrains the rank of the SDP solution for efficiency. Consider the celebrated Goemans-Williamson SDP relaxation for Max-Cut \cite{GW95}:
\begin{align*}
	\label{eq:MC-SDP} \tag{MC-SDP}
	\begin{split}
		&\min_{X \in \sym^{n \times n}} \inangle{A, X} \\
		\text{s.t.} \quad &X_{ii} = 1 \text{ for $i \in [n]$}, \\
		&X \succeq 0.
	\end{split}
\end{align*}
Here $\sym^{n \times n}$ denotes the set of all symmetric $n \times n$ matrices and the cost matrix $A \in \sym^{n \times n}$ is typically the adjacency matrix of a weighted graph. (Note that \eqref{eq:MC-SDP} can also be reformulated as a maximization SDP where the cost matrix is the Laplacian matrix $L = D - A$, where $D$ is the diagonal degree matrix. The two formulations are equivalent since changing the diagonal of the cost matrix just corresponds to adding a constant to the objective value at each feasible point.) The above Goemans-Williamson SDP also gives a natural semidefinite programming relaxation for the Grothendieck problem~\cite{alon2004approximating}, quadratic programming~\cite{Nesterov98, charikarwirth2004}, variants of community detection problems~\cite{Abbe2017}, and other combinatorial optimization problems; in many of these settings the matrix $A$ may also have negative entries.  

Instead of maintaining a solution $X$ in $\sym^{n \times n}$, the Burer-Monteiro method explicitly maintains a low-rank solution of the form $X=YY^\top$ where $Y \in \R^{n \times p}$, and aims to solve the following optimization problem:
\begin{align*}
	\label{eq:MC-BM} \tag{MC-BM}
	\begin{split}
		&\min_{Y \in \R^{n \times p}} \inangle{A, YY^\top} \\
		\text{s.t.} \quad &||Y_i||^2 = 1 \text{ for $i \in [n]$},
	\end{split}
\end{align*}
where $Y_i \in \R^p$ denotes the $i$th row of $Y$, taken as a column vector, and $||\cdot||$ denotes the Euclidean norm. We may also denote the objective of \eqref{eq:MC-BM} as $\obj(Y) := \inangle{A, YY^\top}$. We denote the feasible region as
\begin{align*}
	\MMC_{n, p} := \inbraces{Y \in \R^{n \times p} : ||Y_i||^2 = 1 \text{ for $i \in [n]$} }.
\end{align*}
(When $n$ is clear from the context or unimportant, we may just write $\MMC_p$.) Note that $\MMC_{n, p}$ is a product manifold since it can be viewed as a Cartesian product of $n$ unit spheres in $\R^p$.

This formulation yields significant memory savings when $p \ll n$ by storing $Y$ instead of $X$ and has the additional advantage of dropping the positive semidefiniteness constraint. On the other hand, \eqref{eq:MC-BM} is a non-convex constrained optimization problem. Local optimization methods like Riemannian gradient descent and other heuristics for manifold optimization or constrained optimization are used to solve this non-convex problem with surprisingly good empirical results ~\cite{BM03, BM05}.
This motivates the following question:

\vspace{5pt}
\qquad \qquad {\em When does the Burer-Monteiro method converge to a globally optimal solution? }
\vspace{5pt}
% Say something about the heuristics...

\begin{figure}
	\centering

	\scriptsize
	\begin{tikzpicture}[scale=1.3]
		\tikzstyle{subj} = [circle, minimum width=4pt, fill, inner sep=0pt]

		\draw[thick] (0, 0) to (12, 0);
		%\node[label=left:$p$] at (0, 0) {};

		\node[subj, label=below:{\normalsize{$p = 0$}}] at (0, 0) {};
		\node[subj, label=below:{\normalsize{$p \approx \sqrt{2n}$}}] at (4, 0) {};
		\node[subj, label=below:{\normalsize{$p = \frac{n}{2}$}}] at (8, 0) {};
		\node[subj, label=below:{\normalsize{$p = n$}}] at (12, 0) {};

		\draw[dashed] (0, 0) to (0, 1.8);
		\draw[dashed] (4, 0) to (4, 1.8);
		\draw[dashed] (4, -0.75) to (4, -1.4);
		\draw[dashed] (8, 0) to (8, 1.8);
		\draw[dashed] (12, 0) to (12, 1.6);
		\draw[dashed] (12, -0.5) to (12, -1.4);

		\node[] at (8, -1) {\cite{Bar95, Pat98} Above Barvinok-Pataki bound:};
		\node[] at (8, -1.2) {\eqref{eq:MC-BM} optimal value same as \eqref{eq:MC-SDP} optimal value};
		%\node[] at (8, -1.6) {};

		\node[] at (10, 0.65) {\cite{BVB18} Global convergence guarantee};
		%\node[] at (10, 0.7) {convergence guarantee};

		\node[] at (2, 0.7) {\cite{WW20} No possible global convergence};
		%\node[] at (2, 0.9) {convergence guarantee};
		\node[] at (2, 0.5) {guarantee for generic instances};

		\node[] at (6, 0.7) {\cite{BVB18} Global convergence guarantee };
		%\node[] at (6, 0.9) { convergence guarantee for };
		\node[] at (6, 0.5) {for generic instances};

		\node[] at (6, 1.6) {\textbf{This work:} Spurious minima exist};
		\node[] at (6, 1.4) {(no possible global convergence};
		\node[] at (6, 1.2) { guarantee)};
	\end{tikzpicture}
	\normalsize

	\caption{Combined with existing guarantees, our work settles the question of the existence of spurious local minima for \eqref{eq:MC-BM} in all ranges of the rank \(p\).}
	\label{fig:existence-in-all-ranges-of-rank}
\end{figure}
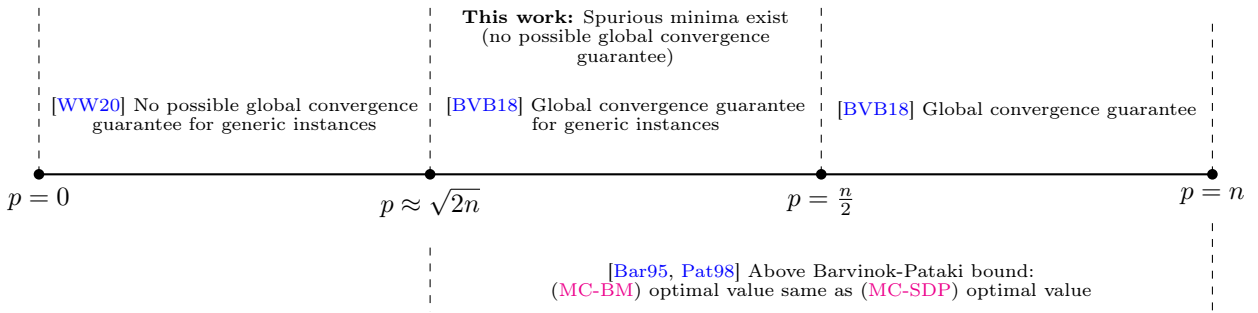

This question has attracted much recent interest on the theoretical front. It is known that when the rank bound $p$ of the SDP solution is above the so-called \textit{Barvinok-Pataki bound} $p \ge \sqrt{2n}$ (more formally, $p$ satisfies $\frac{p (p + 1)}{2} > n$), a globally optimal solution of rank at most $p$ is guaranteed to exist \cite{Bar95, Pat98}. Moreover, below this bound, there exist instances for which \eqref{eq:MC-BM} has a spurious local minimum \cite{WW20}. When the rank bound $p$ is above the Barvinok-Pataki bound, Boumal, Voroninski, and Bandeira \cite{BVB16, BVB18} showed that %above the Barvinok-Pataki bound, %if $\frac{p(p + 1)}{2} > n$ (later slightly improved to $\frac{p(p + 1)}{2} + p > n$ in \cite{WW20}), then 
for {\em generic} instances,\footnote{By generic instances, we mean the guarantee holds for all cost matrices $A \in \sym^{n \times n}$ except a set of zero measure.} any second-order critical point $Y$ of \eqref{eq:MC-BM} is globally optimal for \eqref{eq:MC-BM}, and thus $YY^\top$ is optimal for \eqref{eq:MC-SDP} (their result also applies for a broad class of SDPs with equality constraints).
Under such conditions, known algorithms converge to second-order critical points of \eqref{eq:MC-BM}, and polynomial-time convergence guarantees can be shown for smoothed instances~\cite{BBJN18, PJB18, CM21}.

On the other hand, for general cost matrices $A$, the best known bound only guarantees convergence to a global optimum of the Burer-Monteiro method when $p>\frac{n}{2}$ ~\cite[Cor. 5.11]{BVB18}.
For general cost matrices $A$ of \eqref{eq:MC-BM}, there is a large range for the rank bound $p \in [\sqrt{2n}, n/2]$ where we do not know whether the Burer-Monteiro method works.
To the best of our knowledge, it was open whether there exists {\em any }instance (not specific to Max-Cut) where the Burer-Monteiro method fails above the Barvinok-Pataki bound. In fact, the authors of \cite{BVB18} pose this question:

\vbox{ %makes sure that this block is not split across pages
	\begin{displayquote}
		\textit{``It remains unclear whether or not a zero-measure set of cost matrices must be excluded. Resolving this question is key to gaining a deeper understanding of the relationship between \eqref{eq:MC-SDP} and \eqref{eq:MC-BM}.''}
		\begin{flushright}
			---End of Section 6 in \cite{BVB18}, mutatis mutandis
		\end{flushright}
	\end{displayquote}
}

Our main theorem resolves this question by constructing an instance of \eqref{eq:MC-BM} with a spurious local minimum (i.e., a local minimum which is not globally optimal) when $p$ is as large as $\Theta(n)$ (note that a spurious local minimum is also a spurious second order critical point).  This result is contextualized in Figure \ref{fig:existence-in-all-ranges-of-rank}.

\begin{theorem}
	\label{thm:main-result}
	For any $n \ge 4$ and $2 \le p \le n / 2$, there exist cost matrices $A$ for which the associated instance of \eqref{eq:MC-BM} has a spurious local minimum.
\end{theorem}

Combining Theorem \ref{thm:main-result} with other existing results yields a clearer understanding of the optimization landscape and a characterization of the range of $p$ for which \eqref{eq:MC-BM} can have spurious local minima. Furthermore, our result justifies the use of beyond worst-case paradigms like smoothed analysis to obtain global convergence guarantees for the Burer-Monteiro method. (See also Appendix \ref{app:prior-work} for a further discussion of prior work.)

\paragraph{Outline of paper.}

We begin with Section \ref{sec:preliminaries} (Preliminaries) which provides necessary background for the construction of spurious local minima in Theorem \ref{thm:main-result}.
In Section \ref{sec:main-results}, we list the main lemmas needed for Theorem \ref{thm:main-result} and show how Theorem \ref{thm:main-result} follows. Sections \ref{sec:criticality-proofs}--\ref{sec:bad-matrices-constructions} contain proofs of these lemmas. In particular, Section \ref{sec:criticality-proofs} contains the proofs that our construction yields a spurious first and second-order critical point. Section \ref{sec:local-minimum} contains the proof that our construction yields a spurious local minimum. Section \ref{sec:bad-matrices-constructions} contains constructions of a set of matrices needed in the proof of Theorem \ref{thm:main-result}. Section \ref{sec:experiments} contains experiments, and we present potential follow-up directions in Section \ref{sec:conclusion} (Conclusion).

\section{Preliminaries}
\label{sec:preliminaries}

Section \ref{subsec:Riemannian-derivatives} gives an overview of the Riemannian geometry of \eqref{eq:MC-BM}. In Section \ref{subsec:necessary-sufficient-conditions}, we use this geometry to give necessary conditions for local minimality and a characterization of global optimality. Section \ref{subsec:Riemannian gradient descent} provides an overview of Riemannian gradient descent, which is key to proving local minimality in Theorem \ref{thm:main-result} (see Section \ref{sec:local-minimum} for details). Section \ref{subsec:other-defs} contains definitions (including classes of matrices) used in the construction for Theorem \ref{thm:main-result}.

\subsection{Riemannian derivatives}
\label{subsec:Riemannian-derivatives}

$\MMC_p$ is a smooth embedded submanifold of $\R^{n \times p}$ \cite[Prop. 1.2]{BVB18}, and is furthermore an embedded Riemannian submanifold of $\R^{n \times p}$ \cite[Def. 3.55]{Bou22} if we equip the linearizations of $\MMC_p$ at each point (known as the \textit{tangent spaces}---defined below) with (a restriction of) the inner product $\inangle{\cdot, \cdot}$ on $\R^{n \times p}$.

\begin{proposition}[Tangent space {\cite[Lem. 2.1]{BVB18}}]
	\label{def:MC-BM-tangent-space}
	The tangent space to $\MMC_p$ at $Y \in \MMC_p$, denoted $\T_Y \MMC_p$, is the following subspace of $\R^{n \times p}$:
	\begin{align*}
		\T_Y \MMC_p = \inbraces{U \in \R^{n \times p} : \inangle{Y_i, U_i} = 0 \text{ for $i \in [n]$}}.
	\end{align*}
	Here, $Y_i \in \R^p$, $U_i \in \R^p$ denote the $i$th rows of $Y$ and $U$ respectively, taken as column vectors. (In other words, $\T_Y \MMC_p$ is the space of matrices row-wise orthogonal to $Y$.) We call an element of $\T_Y \MMC_p$ a tangent vector (at $Y$).
\end{proposition}

The Riemannian gradient at $Y \in \MMC_p$, $\grad \obj(Y)$, is the orthogonal projection of the classical Euclidean gradient at $Y$ onto $\T_Y \MMC_p$, which yields the following expression (see Appendix \ref{app:toolbox}):
\begin{align}
	\label{eq:nu}
	\grad \obj(Y) = 2 (A - \diag(\nu)) Y, \quad \text{where } \nu_i := \sum_{j = 1}^n A_{ij} \inangle{Y_i, Y_j}, \text{for all $i \in [n]$}.
\end{align}
(Note that $\nu \in \R^n$ is a function of $Y$, although we write $\nu$ instead of, e.g., $\nu(Y)$ when $Y$ is clear from context.) We take \eqref{eq:nu} to be the definition of $\nu$ from now on.

The Riemannian Hessian of $\obj$ at $Y \in \MMC_p$, $\Hess \obj(Y)$, is a linear, symmetric map from $\T_Y \MMC_p$ to $\T_Y \MMC_p$ given by the classical differential of (a smooth extension of) $\grad \obj(Y)$, projected to the tangent space \cite[Cor. 5.16]{Bou22}. This yields (see Appendix \ref{app:toolbox}):
\begin{align*}
	\Hess \obj(Y) [U] = 2 \Proj_Y \bigg( (A - \diag(\nu))U \bigg)
\end{align*}
for $U \in \T_Y \MMC_p$, where the linear map $\Proj_Y : \R^{n \times p} \to \T_Y \MMC_p$ denotes the orthogonal projector onto $\T_Y \MMC_p \subseteq \R^{n \times p}$, i.e., $\Proj_Y (Z) = \argmin_{U \in \T_Y \MMC_p} \norm{U - Z}$. In what follows, $\Hess \obj(Y)$ will only appear as part of a quadratic form, yielding the following cleaner expression:
\begin{align*}
	\inangle{ \Hess \obj(Y) [U], U } = \inangle{ 2 \Proj_Y \inparen{ (A - \diag(\nu))U }, U }
	= 2 \inangle{A - \diag(\nu), UU^\top},
\end{align*}
for $U \in \T_Y \MMC_p$, where we used the fact that $\inangle{\Proj_Y (Z), U} = \inangle{Z, U}$ for any $Z \in \R^{n \times p}, U \in \T_Y \MMC_p$.

\subsection{Necessary and sufficient conditions}
\label{subsec:necessary-sufficient-conditions}

Recall the definition of a local (and global) minimum:

\begin{definition}[Local/global minimum]
	Consider the program
	\begin{align*}
		\min_{x \in D} f(x)
	\end{align*}
	where $D \subseteq \R^d$. $x \in D$ is a local minimum if there exists $\epsilon > 0$ such that if $y \in D$ and $||y - x|| < \epsilon$, we have $f(x) \le f(y)$. $x \in D$ is a global minimum if $f(x) \le f(y)$ for all $y \in D$.
\end{definition}

The following are standard necessary conditions for local optimality, and correspond to the first and second-order critical point criteria that need to be satisfied by any local minimum of \eqref{eq:MC-BM} (see \cite[Prop 2.4]{BVB18}).

\begin{proposition}[First-order critical point {\cite[Def. 2.3]{BVB18}} and {\cite[Prop. 3]{WW20}}]
	\label{prop:MC-BM-first-order-crit}
	$Y \in \MMC_p$ is a first-order critical point for \eqref{eq:MC-BM} if and only if $\grad \obj(Y) = 2 (A - \diag(\nu)) Y = 0$. Equivalently, $Y \in \MMC_p$ is a first-order critical point if and only if there exists $\lambda \in \R^n$ such that $\inparen{A - \diag(\lambda)} Y = 0$. If such a $\lambda$ exists, it is unique and equal to $\nu$ given by \eqref{eq:nu}.
\end{proposition}

\begin{proposition}[Second-order critical point {\cite[Prop. 4]{WW20}}]
	\label{prop:MC-BM-second-order-crit}
	A first-order critical point $Y \in \MMC_p$ is additionally a second-order critical point if and only if
	\begin{align*}
		\inangle{ \Hess \obj(Y)[U], U} = 2 \inangle{A - \diag(\nu), UU^\top} \ge 0
	\end{align*}
	for all $U \in \T_Y \MMC_p$, and where $\nu$ is given in \eqref{eq:nu}. (This is equivalent to $\Hess \obj(Y) \succeq 0$.)
\end{proposition}

We say a critical point or local minimum is \textit{spurious} if it is not globally optimal. Finally, we characterize which first-order critical points of \eqref{eq:MC-BM} are globally optimal. Since second-order critical points and local minima are also first-order critical points, this also provides a characterization of optimality for them.

\begin{proposition}[Characterization of optimality for first-order critical points]
	\label{prop:char-first-order-opt}
	A first-order critical point $Y$ of \eqref{eq:MC-BM} is globally optimal if and only if $A - \diag(\nu) \succeq 0$, where $\nu$ is given in \eqref{eq:nu}.
\end{proposition}

While not framed precisely in this way, Proposition \ref{prop:char-first-order-opt} follows directly from prior work (see Appendix \ref{app:toolbox} for details). The proof involves a comparison between the criticality conditions of \eqref{eq:MC-BM} and the Karush–Kuhn–Tucker (KKT) conditions of \eqref{eq:MC-SDP}.

% The ``if'' direction involves showing that the given assumptions imply that the primal-dual pair $(YY^\top, \nu)$ satisfies the Karush–Kuhn–Tucker (KKT) conditions of \eqref{eq:MC-SDP}, implying (after a few simple additional steps) that $Y$ is globally optimal for \eqref{eq:MC-BM}.

% One can show that if $Y$ if a first-order critical point and $A - \diag(\nu) \succeq 0$, then the primal-dual pair $(YY^\top, \nu)$ satisfies the Karush–Kuhn–Tucker (KKT) conditions of \eqref{eq:MC-SDP}, implying $YY^\top$ is optimal for \eqref{eq:MC-SDP} since \eqref{eq:MC-SDP} is convex. As \eqref{eq:MC-SDP} is effectively a relaxation of \eqref{eq:MC-BM}, we conclude that $Y$ is optimal for \eqref{eq:MC-BM}. The other direction is mildly more complicated and relies on the fact that strong duality holds for \eqref{eq:MC-BM} (since it satisfies Slater's condition). % See Appendix~\ref{app:toolbox} for details.

\subsection{Riemannian gradient descent}
\label{subsec:Riemannian gradient descent}

The analogue to gradient descent for optimizing over a smooth manifold is Riemannian gradient descent (see \cite[Ch. 4]{Bou22} for an introduction), which yields analogous analyses and guarantees. It includes as part of its specification a \textit{retraction} \cite[Def. 3.47]{Bou22}. A rectraction on $\MMC_p$ associates to each point $Y \in \MMC_p$ a map $\ret_Y : \T_Y \MMC_p \to \MMC_p$ which converts movement in the tangent space to movement on the manifold $\MMC_p$.  % in a natural way. 
We use the natural \textit{metric projection retraction} \cite[Sec. 5.12]{Bou22}, defined for $Y \in \MMC_p, U \in \T_Y \MMC_p$ as $\ret_Y (U) := \argmin_{Z \in \MMC_p} \norm{(Y + U) - Z}$.
With this definition, it is easy to see that $\ret_Y (U)$ is $Y + U$ followed by a normalization of each row. This yields the following Riemannian gradient descent algorithm for \eqref{eq:MC-BM}:

\textbf{Input:} Initializer $Y^{(0)} \in \MMC_p$, step size $\eta > 0$.

\textbf{For} $t = 0, 1, 2, \dots$
$$Y^{(t + 1)} = \ret_{Y^{(t)}} \inparen{- \eta  \grad \obj (Y^{(t)})}.$$

\subsection{Construction-specific definitions}
\label{subsec:other-defs}

Finally, we give a few miscellaneous technical definitions which will be used in our construction of a spurious local minimum for \eqref{eq:MC-BM}.

\begin{definition}[Axial position]
	\label{def:axial-position}
	We call the matrix
	\begin{align*}
		\yaxial :=
		\begin{bmatrix}
			I_{(n/2)} \\[3pt]
			- I_{(n/2)}
		\end{bmatrix} \in \MMC_{n, (n / 2)}
	\end{align*}
	the axial position, where $I_{n / 2}$ denotes the identity matrix in $\R^{\frac{n}{2} \times \frac{n}{2}}$.
\end{definition}

We use the term ``axial position" because when we view $\MMC_{n, (n / 2)}$ as a Cartesian product of $n$ unit spheres in $\R^{n/2}$, $\yaxial$ corresponds to placing a single unit vector on both the negative and positive sides of each axis in $\R^{n/2}$.

The following sets of matrices will be important in our construction:

\begin{definition}[Pseudo-PD, pseudo-PSD]
	\label{def:pseudo-PSD}
	We say a matrix $M \in \sym^{n \times n}$ is pseudo-PD (``pseudo-positive definite'') if $M[i] \succ 0$ for all $i \in [n]$, where $M[i] \in \sym^{(n - 1) \times (n - 1)}$ denotes the submatrix of $M$ formed by removing the $i$th row and column. Similarly, we say that $M \in \sym^{n \times n}$ is pseudo-PSD (``pseudo-positive semidefinite'') if $M[i] \succeq 0$ for all $i \in [n]$.
\end{definition}

\begin{definition}[Strictly pseudo-PD, strictly pseudo-PSD]
	\label{def:strictly-pseudo-PSD}
	We say a matrix is $M \in \sym^{n \times n}$ is strictly pseudo-PD if it is pseudo-PD but not positive semidefinite. We say $M$ is strictly pseudo-PSD if it is pseudo-PSD but not positive semidefinite. (Note that in both cases we require $M \nsucceq 0$, not $M \nsucc 0$.)
\end{definition}

Clearly every strictly pseudo-PD matrix is also strictly pseudo-PSD, but the converse turns out to be false.

\section{Main Claims and Outline}
\anote{Changed from: Proof of Theorem \ref{thm:main-result}}
\label{sec:main-results}
In this section we discuss the main claims and put them together to prove Theorem \ref{thm:main-result}. We focus on the case where \(n\) is even and $p = n / 2$ and construct cost matrices for which $\yaxial$ is a spurious local minimum, since constructions for $p < n / 2$ can be easily extracted from the former by padding with zeros. Before this, as a warm-up, we characterize those cost matrices for which $\yaxial$ is a spurious first and second-order critical point in the following two propositions. \anote{Added:} First-order and second-order criticality are necessary but not sufficient to establish that $\yaxial$ is a spurious local minima. 
While not strictly necessary for the proof of Theorem \ref{thm:main-result}, Propositions \ref{lem:Yaxial-first-order-crit} and \ref{lem:Yaxial-second-order-crit} have far simpler proofs (see Section \ref{sec:criticality-proofs}) and are interesting in their own right.

\begin{restatable}[First-order critical point characterization for $\yaxial$]{proposition}{firstordercriticality}
    \label{lem:Yaxial-first-order-crit}
	For \eqref{eq:MC-BM} when $p = n / 2$, the axial position $\yaxial$ is a first-order critical point if and only if the cost matrix $A$ takes the form
	\begin{align}
		\label{eq:Yaxial-first-order-obj}
		A =
		\begin{bmatrix}
			B & B \\
			B & B
		\end{bmatrix} + \diag(\alpha)
	\end{align}
	for some $\alpha \in \R^n$ and $B \in \sym^{\frac{n}{2} \times \frac{n}{2}}$. Furthermore, $\yaxial$ is additionally spurious if and only if $B \nsucceq 0$.
\end{restatable}

\begin{restatable}[Second-order critical point characterization for $\yaxial$]{proposition}{secondordercriticality}
	\label{lem:Yaxial-second-order-crit}
	For \eqref{eq:MC-BM} when $p = n / 2$, the axial position $\yaxial$ is a second-order critical point if and only if the cost matrix $A$ takes the form
	\begin{align*}
		% \label{eq:Yaxial-second-order-obj}
		A =
		\begin{bmatrix}
			P & P \\
			P & P
		\end{bmatrix} + \diag(\alpha)
	\end{align*}
	for some $\alpha \in \R^n$ and pseudo-PSD $P \in \sym^{\frac{n}{2} \times \frac{n}{2}}$. Furthermore, $\yaxial$ is additionally spurious if and only if $P$ is strictly pseudo-PSD.
\end{restatable}

It is not surprising that Propositions \ref{lem:Yaxial-first-order-crit} and \ref{lem:Yaxial-second-order-crit} allow you to arbitrarily change the diagonal of the cost matrix. Doing so simply corresponds to adding a constant to the objective value at each point and does not change the geometry of the problem.\footnote{One can easily check that the Riemannian derivatives at any point $Y \in \MMC_p$ remain unchanged since $\nu$ in \eqref{eq:nu} will act as an offset.}

Next, the following lemma provides a sufficient condition for $\yaxial$ to be a spurious local minimum.  This is the most challenging of our results to prove, and we give the proof and discuss the challenges involved in Section \ref{sec:local-minimum}.

\begin{restatable}[Local minimum condition for $\yaxial$]{lemma}{localmin}
	\label{lem:Yaxial-local-min}
	For \eqref{eq:MC-BM} when $p = n / 2$, the axial position $\yaxial$ is a local minimum if the cost matrix $A$ takes the form
	\begin{align}
		\label{eq:Yaxial-local-min-obj}
		A =
		\begin{bmatrix}
			M & M \\
			M & M
		\end{bmatrix} + \diag(\alpha)
	\end{align}
	for some $\alpha \in \R^n$ and pseudo-PD $M \in \sym^{\frac{n}{2} \times \frac{n}{2}}$.
	Furthermore, $\yaxial$ is additionally spurious if $M$ is strictly pseudo-PD.
\end{restatable}

Actualizing Lemma \ref{lem:Yaxial-local-min} to construct a spurious local minimum requires the existence of strictly pseudo-PD matrices, which we posit in the following lemma:

\begin{restatable}[Existence of strictly pseudo-PD matrices]{lemma}{pseudopdexist}
	\label{prop:existence-of-Mbad}
	The set of $k \times k$ strictly pseudo-PD matrices is nonempty for any $k \ge 2$.
\end{restatable}

We provide constructions of strictly pseudo-PD matrices in Section \ref{sec:bad-matrices-constructions}. Our main nonnegative construction takes the form $UU^\top - \epsilon I_k$, where $U \in \R^{k \times (k - 1)}$ is a random matrix. It is shown $(UU^\top)[i] \succ 0$ for all $i$ with high probability. One can show that when $\epsilon > 0$ is sufficiently small, the $(k - 1) \times (k - 1)$ principal submatrices of $UU^\top - \epsilon I_k$ remain positive definite, while $UU^\top - \epsilon I_k \nsucceq 0$ since $UU^\top$ is rank-deficient.

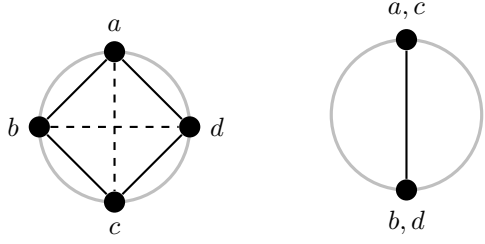
\begin{figure}
	\centering
	\begin{tikzpicture}[scale=0.5]
		\tikzstyle{subj} = [circle, minimum width=8pt, fill, inner sep=0pt]

		\draw[color=gray!50, very thick] (0, 0) circle (2);

		\node[subj, label=above:$a$] (a) at (0, 2) {};
		\node[subj, label=left:$b$] (b) at (-2, 0) {};
		\node[subj, label=below:$c$] (c) at (0, -2) {};
		\node[subj, label=right:$d$] (d) at (2, 0) {};

		\draw[thick] (a) to (b);
		\draw[thick] (b) to (c);
		\draw[thick] (c) to (d);
		\draw[thick] (a) to (d);
		\draw[thick, dashed] (a) to (c);
		\draw[thick, dashed] (b) to (d);
	\end{tikzpicture} \hspace{3em}
	\begin{tikzpicture}[scale=0.5]
		\tikzstyle{subj} = [circle, minimum width=8pt, fill, inner sep=0pt]

		\draw[color=gray!50, very thick] (0, 0) circle (2);

		\node[subj, label=above:{$a,c$}] (a) at (0, 2) {};
		\node[subj, label=below:{$b,d$}] (b) at (0, -2) {};

		\draw[thick] (a) to (b);
	\end{tikzpicture}

	\caption{A spurious minimum and corresponding global minimum for \(n=4, p=2\).}
	%The dashed lines indicate light (lower weight) edges, and the solid lines indicate heavy (higher weight) edges. 
	\label{fig:n4p2localmin}
\end{figure}

We use Lemmas \ref{lem:Yaxial-local-min} and \ref{prop:existence-of-Mbad} to prove Theorem \ref{thm:main-result}:

\begin{proofof}{Theorem \ref{thm:main-result}}
	Lemmas \ref{lem:Yaxial-local-min} and \ref{prop:existence-of-Mbad} imply that for even $n$ such that $n \ge 4$, there exists an instance of \eqref{eq:MC-BM} with a spurious local minimum when $p = n / 2$.
	To construct a spurious local minimum for $p < n/2$, we can simply use the construction for \(n' = 2p\) vertices and rank \(p\).  The entries of the cost matrix that correspond to additional rows of elements of $\MMC_{n, p}$ past \(n'\) can be set to 0, ensuring that these additional rows cannot affect the objective value. (See Appendix \ref{app:extending-to-p<n/2} for details.)
\end{proofof}

One instantiation of our construction for \(n = 4, p = 2\) is illustrated in Figure \ref{fig:n4p2localmin}.  In this visualization, we think of the nonnegative cost matrix as the adjacency matrix of a weighted graph, which is natural for \eqref{eq:MC-BM}.  Each row of \(Y\) specifies the position of one of the vertices on the unit sphere in $\R^p$ (for \(p = 2\), the circle shown in gray).  We illustrate ``heavy'' (higher weight) edges with solid lines, and ``light'' (lower weight) edges with dashed lines.  Intuitively, each edge ``pushes" its endpoints away from each other, with heavier edges pushing harder.  We see by symmetry that this state is in equilibrium (each vertex is pushed equally clockwise and counter-clockwise), so the gradient is 0.  Showing that this instance is indeed a local minimum is more involved, and requires arguing that if the vertices were perturbed slightly, the heavy edges would still approximately cancel out, and the main force on the vertices would be the light edges pushing the pairs \(a, c\) and \(b, d\) back to being diametrically opposite.\footnote{We provide an interactive visualization for the \(p = 2\) and \(p = 3\) cases at\\ \url{https://vaidehi8913.github.io/burer-monteiro}.}

\section{Warmup: proofs of first-order and second-order criticality} 
\label{sec:criticality-proofs}
\anote{Removed: (Propositions \ref{lem:Yaxial-first-order-crit} and \ref{lem:Yaxial-second-order-crit}). Added it to text in the start of the section.}
%\anote{Does the reader need to read this section, or move this to the appendix? I really like Section 5, and the way it is structured; I'd like the reader to get to this sooner. }
\anote{Please read this text and verify.}
In this section we give the proofs of Propositions \ref{lem:Yaxial-first-order-crit} and \ref{lem:Yaxial-second-order-crit}, which characterize cost matrices for which $\yaxial$ is first order critical and second order critical respectively. These propositions together motivate the pseudo-PSD property of cost matrices that are crucial in our construction. The reader is welcome to skip ahead to Section~\ref{sec:local-minimum} for the proof of the main lemma about $\yaxial$ being a spurious local minimum.

% Both proofs are essentially a matter of checking the necessary conditions given in Section \ref{subsec:necessary-sufficient-conditions}.

\anote{Since there is no page restriction, should we paste the claim here again?} \vnote{resolved}

\firstordercriticality*

\begin{proofof}{Proposition \ref{lem:Yaxial-first-order-crit}}
	We first prove the first half of Proposition \ref{lem:Yaxial-first-order-crit}, which characterizes when $\yaxial$ is a first-order critical point.
	If $A$ takes the form \eqref{eq:Yaxial-first-order-obj} for some $\alpha \in \R^n$ and $B \in \sym^{\frac{n}{2} \times \frac{n}{2}}$, we can set $\lambda \gets \alpha$ where $\lambda$ is our choice for the multiplier from Proposition \ref{prop:MC-BM-first-order-crit}. Observe then that $(A - \diag(\lambda)) \yaxial = 0$,
	implying $\yaxial$ is indeed a first-order critical point. For the other direction, suppose that $\yaxial$ is a first-order critical point with associated multiplier $\lambda$ (from Proposition \ref{prop:MC-BM-first-order-crit}), and consider the matrix $S := A - \diag(\lambda) \in \sym^{n \times n}$, which can be expressed in the block form
	\begin{align*}
		S =
		\begin{bmatrix}
			S_1      & S_2 \\
			S_2^\top & S_3
		\end{bmatrix}
	\end{align*}
	for some $S_1, S_3 \in \sym^{\frac{n}{2} \times \frac{n}{2}}$ and $S_2 \in \R^{\frac{n}{2} \times \frac{n}{2}}$. Then $S \yaxial = 0$ implies $S_1 - S_2 = 0$ and $S_2^\top - S_3 = 0$, and thus $S_1 = S_2 = S_3$. Thus, $A = S + \diag(\lambda)$ indeed takes the form \eqref{eq:Yaxial-first-order-obj}.

	Now we prove the second half of Proposition \ref{lem:Yaxial-first-order-crit}: the characterization of when $\yaxial$ is a spurious first-order critical point. Supposing that the cost matrix takes the form \eqref{eq:Yaxial-first-order-obj} for some $\alpha \in \R^n$ and $B \in \sym^{\frac{n}{2} \times \frac{n}{2}}$ (as we've shown is necessary for $\yaxial$ to be a first-order critical point), we show that $\yaxial$ is additionally spurious if and only if $B \nsucceq 0$. Recall from above that the unique multiplier $\lambda \in \R^n$ (from Proposition \ref{prop:MC-BM-first-order-crit}) associated with $\yaxial$ is precisely $\alpha$. Then it follows from Proposition \ref{prop:char-first-order-opt} that $\yaxial$ is spurious if and only if
	\begin{align*}
		A - \diag(\lambda) =
		\begin{bmatrix}
			B & B \\
			B & B
		\end{bmatrix}
	\end{align*}
	is not positive semidefinite. (Recall from Proposition \ref{prop:MC-BM-first-order-crit} that $\nu = \lambda$ at a first-order critical point.) We claim
	\begin{align}
		\label{eq:block-matrix-PSD}
		\begin{bmatrix}
			B & B \\
			B & B
		\end{bmatrix} = \begin{bmatrix}
			1 & 1 \\
			1 & 1
		\end{bmatrix} \otimes B \nsucceq 0 \quad \iff  \quad B \nsucceq 0,
	\end{align}
	where $\otimes$ denotes the Kronecker product. Indeed, this follows because the spectrum of $F \otimes G$, denoted $\sigma(F \otimes G)$, for two square, real matrices $F, G$ is given by
	\begin{align*}
		\sigma(F \otimes G) = \inbraces{\lambda \mu : \lambda \in \sigma(F), \mu \in \sigma(G)}.
	\end{align*}
	This, combined with the fact that $\begin{bmatrix}
			1 & 1 \\
			1 & 1
		\end{bmatrix} \succeq 0$, implies \eqref{eq:block-matrix-PSD}.
\end{proofof}

As for the second proposition:

\secondordercriticality*

\begin{proofof}{Proposition \ref{lem:Yaxial-second-order-crit}}
	We first prove the first half of Proposition \ref{lem:Yaxial-second-order-crit}, which characterizes when $\yaxial$ is a second-order critical point.
	Since any second-order critical point is also a first-order critical point, the first half of Proposition \ref{lem:Yaxial-first-order-crit} implies it is necessary for the cost matrix $A$ to take the form \eqref{eq:Yaxial-first-order-obj} for some $\alpha \in \R^n$ and $B \in \sym^{\frac{n}{2} \times \frac{n}{2}}$ for $\yaxial$ to be a second-order critical point. We will show that $\yaxial$ is additionally a second-order critical point if and only if the matrix $B$ from \eqref{eq:Yaxial-first-order-obj} is pseudo-PSD.

	To this end, recall from the proof of Proposition \ref{lem:Yaxial-first-order-crit} that the unique multiplier $\lambda$ associated with $\yaxial$ when $A$ takes the form \eqref{eq:Yaxial-first-order-obj} is precisely $\alpha$. Then, writing
	\begin{align}
		\label{eq:S}
		S := A - \diag(\lambda) = \begin{bmatrix}
			B & B \\
			B & B
		\end{bmatrix},
	\end{align}
	clearly the condition for the second-order criticality of $\yaxial$ (Proposition \ref{prop:MC-BM-second-order-crit}) is equivalent to
	\begin{align}
		\label{eq:second-order-crit-cond}
		\inangle{S, UU^\top} \ge 0 \quad \text{for $U \in \T_\yaxial \MMC_{n/2}$}.
	\end{align}
	Define
	\begin{align*}
		\off := \inbraces{ G \in \R^{\frac{n}{2} \times \frac{n}{2}} : \diag(G) = 0 }
	\end{align*}
	to be the subspace of $\R^{\frac{n}{2} \times \frac{n}{2}}$ consisting of matrices with zeros on their diagonals. Note then that $\T_\yaxial \MMC_{n/2}$ is precisely
	\begin{align}
		\label{eq:tangent-space-Yaxial}
		\T_\yaxial \MMC_{n/2} = \inbraces{
			\begin{bmatrix}
				U_1 \\
				U_2
			\end{bmatrix}
			: U_1, U_2 \in \off
		}.
	\end{align}
	In other words, $\T_\yaxial \MMC_{n/2}$ is the set of all matrices $U \in \R^{n \times \frac{n}{2}}$ such that $U_{ii} = U_{(n / 2 + i), i} = 0$ for all $i \in [n / 2]$, and the other entries are completely arbitrary. Observe that \eqref{eq:S} and \eqref{eq:tangent-space-Yaxial} imply the second-order criticality condition \eqref{eq:second-order-crit-cond} is equivalent to
	\begin{align}
		\label{eq:second-order-crit-cond-2}
		\inangle{B, (U_1 + U_2)(U_1 + U_2)^\top} \ge 0 \quad \text{for $U_1, U_2 \in \off$}.
	\end{align}
	Next, note that \eqref{eq:second-order-crit-cond-2} (and therefore \eqref{eq:second-order-crit-cond}) is equivalent to
	\begin{align}
		\label{eq:second-order-crit-cond-3}
		\inangle{B, G G^\top} \ge 0 \quad \text{for $G \in \off$}
	\end{align}
	since $\off$ is closed under addition.

	Thus, we have shown at this point that $\yaxial$ is a second-order critical point if and only if \eqref{eq:second-order-crit-cond-3} holds. Now let $T_i$ for $i \in [n / 2]$ denote the $(n / 2 - 1)$-dimensional subspace of $\R^{n / 2}$ obtained by fixing the $i$th entry to be 0 and letting all other entries vary arbitrarily. Observe that
	\begin{align*}
		\inbraces{GG^\top : G \in \off} = \inbraces{ \sum_{k = 1}^{n / 2} v_k v_k^\top : v_i \in T_i \text{ for all $i \in [n / 2]$} }.
	\end{align*}
	Thus, we can reexpress the second-order criticality condition \eqref{eq:second-order-crit-cond-3} as follows:
	\begin{align}
		     & \inangle{B, G G^\top} \ge 0 \quad \text{for $G \in \off$}  \nonumber                                                           \\
		\iff & \inangle{B, \sum_{k = 1}^{n / 2} v_k v_k^\top} \ge 0 \quad \text{for $v_1, \in T_1, \dots, v_{n / 2} \in T_{n / 2}$} \nonumber \\
		\iff & \inangle{B, v v^\top} \ge 0 \quad \text{for $v \in T_1 \cup \dots \cup T_{n / 2}$}. \label{eq:second-order-crit-cond-4}
	\end{align}
	\eqref{eq:second-order-crit-cond-4} is equivalent to $B$ being pseudo-PSD.

	As for the second half of Proposition \ref{lem:Yaxial-second-order-crit}, the characterization of when $\yaxial$ is a spurious second-order critical point, this follows immediately from the first half of Proposition \ref{lem:Yaxial-second-order-crit} and the characterization of when $\yaxial$ is a spurious first-order critical point from Proposition \ref{lem:Yaxial-first-order-crit} (since all second-order critical points are also first-order critical points).
\end{proofof}

\section{Proof of local minimality (Lemma~\ref{lem:Yaxial-local-min})}
\label{sec:local-minimum}
\anote{I like the way this section is structured! }

In this section we give the proof of Lemma \ref{lem:Yaxial-local-min}. Section \ref{subsec:challenges-sublemmas-Lemma-3} discusses the challenges involved, states the necessary sublemmas, and concludes with the proof of Lemma \ref{lem:Yaxial-local-min}. Section \ref{subsec:conv-to-same-obj-value} gives the proof of the most important sublemma (Lemma \ref{lem:conv-to-point-with-same-obj}).  

\subsection{Challenges, key sublemmas, and the proof of Lemma \ref{lem:Yaxial-local-min}}
\label{subsec:challenges-sublemmas-Lemma-3}

\paragraph{Challenges.}
Unfortunately, arguing about the value of the objective function at some point \(Y\) near $\yaxial$ is challenging, and classical techniques for proving that $\yaxial$ is a local minimum fail. For example, \cite{WW20}, which constructs spurious local minima for \eqref{eq:MC-BM} when $p < \sqrt{2n}$, similarly first constructs spurious second-order critical points and then proves that they are additionally local minima. However, their proof follows because their spurious second-order critical points are non-degenerate \cite[Def. 3]{WW20}, which corresponds to the rank of the Riemannian Hessian being sufficiently high. We show in Appendix \ref{app:challenges} (Proposition \ref{prop:degenerate-over-Barvinok-Pataki}) that for \textit{any} instance of \eqref{eq:MC-BM} when $p \ge \sqrt{2n}$, \textit{all} spurious second-order critical points are degenerate, meaning this approach will not work. Furthermore, there is no hope of using the positivity of a higher-order Riemannian derivative (e.g., the fourth derivative) to prove that $\yaxial$ is a local minimum, since it can be shown that all higher-order derivatives are degenerate. (See Appendix \ref{app:challenges} for further discussion of these challenges.)

\paragraph{Overview and key sublemmas.} Thus, we provide a novel approach involving Riemannian gradient descent (Section \ref{subsec:Riemannian gradient descent}). For the first sublemma below, recall that a neighborhood of a point $Y \in \MMC_p$ is a set of the form $\{ Y' \in \MMC_p : ||Y - Y'|| < \epsilon \}$ for some $\epsilon > 0$. The proof is given in Section \ref{subsec:conv-to-same-obj-value}.

\begin{lemma}[Convergence to a point with the same objective value]
	\label{lem:conv-to-point-with-same-obj}
	In the setting of Lemma \ref{lem:Yaxial-local-min} with pseudo-PD $M \in \sym^{\frac{n}{2} \times \frac{n}{2}}$,
	there exists a neighborhood $N \subseteq \MMC_{n / 2}$ of $\yaxial$ and $\etaprime > 0$ depending only on the instance of \eqref{eq:MC-BM} such that if you initialize Riemannian gradient descent (as specified in Section \ref{subsec:Riemannian gradient descent}) with any step size $\eta < \etaprime$ at any point $\yinitial \in N$, it converges to a point $\yconv$ such that $\obj(\yconv) = \obj(\yaxial)$.
\end{lemma}

Note that Lemma \ref{lem:conv-to-point-with-same-obj} does not imply convergence to $\yaxial$ itself; this is not actually true due to the degeneracy mentioned above. Instead, we show in the proof of Lemma \ref{lem:conv-to-point-with-same-obj} that $\yconv$ is an \textit{antipodal configuration}, i.e., it takes the form $\begin{bmatrix}
		G \\
		-G
	\end{bmatrix} \in \MMC_{n / 2}$ for some $G \in \R^{\frac{n}{2} \times \frac{n}{2}}$. Such antipodal configurations (of which $\yaxial$ is one) all have the same objective value and correspond in particular to ``flat'' directions from $\yaxial$. For a given $\yinitial$, it is not a priori clear which of these antipodal configurations it will converge to, so we argue convergence to \textit{some} antipodal configuration.

In the proof of Lemma \ref{lem:conv-to-point-with-same-obj} in Section \ref{subsec:conv-to-same-obj-value}, we track convergence to an antipodal configuration via a potential $\Phi: \MMC_{n / 2} \to \R_{\ge 0}$, where $\Phi \inparen{\begin{bmatrix}
			G_1 \\
			G_2
		\end{bmatrix}} := ||G_1 + G_2||^2$. (Here, $G_1, G_2 \in \R^{\frac{n}{2} \times \frac{n}{2}}$.) Clearly $\Phi$ is 0 if and only if the input is antipodal. We show that $\Phi$ decreases geometrically over the iterations of Riemannian gradient descent.

Next, we show via a smoothness argument that with sufficiently small step size, the objective is nonincreasing over the iterations of Riemannian gradient descent. The proof is given in Appendix \ref{app:nonincreasing}.

\begin{lemma}[$\obj$ is nonincreasing]
	\label{lem:obj-nonincreasing}
	There exists $\etatilde > 0$ depending only on the instance of \eqref{eq:MC-BM} such that a single iteration of Riemannian gradient descent (as specified in Section \ref{subsec:Riemannian gradient descent}) with any step size $\eta < \etatilde$ cannot increase the objective value, regardless of the starting point.
\end{lemma}

Lemmas \ref{lem:conv-to-point-with-same-obj} and \ref{lem:obj-nonincreasing} imply $\obj(\yinitial) \le \obj(\yconv) = \obj(\yaxial)$, and as a result the neighborhood $N$ in Lemma \ref{lem:conv-to-point-with-same-obj} certifies that $\yaxial$ is a local minimum. The details follow:

\localmin*

\begin{proofof}{Lemma \ref{lem:Yaxial-local-min}}
	We first show that if $M \in \sym^{\frac{n}{2} \times \frac{n}{2}}$ is pseudo-PD, then $\yaxial$ is a local minimum. We claim that the neighborhood $N$ from Lemma \ref{lem:conv-to-point-with-same-obj} certifies that $\yaxial$ is a local minimum. Indeed, let $V \in N$, and initialize Riemannian gradient descent at $V$ with step size $\eta < \min \inbraces{\etaprime, \etatilde}$, with $\etaprime, \etatilde$ defined as in Lemmas \ref{lem:conv-to-point-with-same-obj}, \ref{lem:obj-nonincreasing}. Per Lemma \ref{lem:conv-to-point-with-same-obj}, we know that Riemannian gradient descent will converge to a point $\vconv$ such that $\obj(\yaxial) = \obj(\vconv)$.

	% Recall that a $C^1$ function over a compact set is Lipschitz, so $\obj$ is Lipschitz over $\MMC_{n / 2}$. % Actually I think just continuity is enough here.
	Since $\obj$ is continuous, convergence in iterates translates to convergence in objective values, and thus the nonincreasing nature of Riemannian gradient descent from Lemma \ref{lem:obj-nonincreasing} implies $\obj(\vconv) \le \obj(V)$. Then $\obj(\yaxial) = \obj(\vconv) \le \obj(V)$, and we are done.

	Now we show that $\yaxial$ is additionally spurious if $M$ is strictly pseudo-PD. Indeed, this follows immediately from the last line of Proposition \ref{lem:Yaxial-first-order-crit}. (Recall that any local minimum is also a first-order critical point.)
\end{proofof}

\subsection{Proof of key sublemma: convergence to a point with the same objective value (Lemma \ref{lem:conv-to-point-with-same-obj})}
\label{subsec:conv-to-same-obj-value}
In this section we prove Lemma \ref{lem:conv-to-point-with-same-obj}, which is the main sublemma behind the proof of Lemma \ref{lem:Yaxial-local-min} in the last section. Taken together, the proof of Lemma \ref{lem:conv-to-point-with-same-obj} is by far the longest in this paper and will itself utilize several sublemmas given in this section (with some additional very minor claims proven in Appendix \ref{app:minor-claims}). See the very end of this section for the proof of Lemma \ref{lem:conv-to-point-with-same-obj} itself.

\paragraph{Important setup for this section.} Throughout Section \ref{subsec:conv-to-same-obj-value}, we assume we are in the setting of Lemma \ref{lem:Yaxial-local-min} with pseudo-PD $M \in \sym^{\frac{n}{2} \times \frac{n}{2}}$. Furthermore, we assume for simplicity that $\alpha = 0$. This is without loss of generality because due to the feasibility constraint of \eqref{eq:MC-BM}, shifting the diagonal entries of the cost matrix just corresponds to adding the same constant to the objective value of each feasible point. In particular, it is easy to check that changing $\alpha$ does not affect the geometry of the problem, i.e., the Riemannian derivatives at any point $Y \in \MMC_p$ remain unchanged. As a result of these assumptions, the cost matrix in this section always takes the form
\begin{align}
	\label{eq:cost-mat-setup}
	A =
	\begin{bmatrix}
		M & M \\
		M & M
	\end{bmatrix}
\end{align}
for some pseudo-PD $M \in \sym^{\frac{n}{2} \times \frac{n}{2}}$. Finally, $p$ is always $n / 2$ in this section. (We may sometimes write $p$ instead of $n / 2$ for shorthand.)

To start, the following sublemma, which was described briefly in words in Section \ref{sec:local-minimum}, provides the backbone of the argument. Recall once again that by a neighborhood of $Y \in \MMC_p$, we mean a set of the form $\inbraces{Y' \in \MMC_p : ||Y - Y'|| < \epsilon}$ for some $\epsilon > 0$, where $||\cdot||$ as always denotes the Euclidean (or equivalently Frobenius) norm.

\begin{lemma}[Decrease in the potential $\Phi$]
	\label{lem:decrease-in-potential}
	Let the potential $\Phi : \MMC_{n / 2} \to \R_{\ge 0}$ be defined as $\Phi \inparen{\begin{bmatrix}
				G_1 \\
				G_2
			\end{bmatrix}} := ||G_1 + G_2||^2$, where $G_1, G_2 \in \R^{\frac{n}{2} \times \frac{n}{2}}$. Then there exists a neighborhood $\neigh$ of $\yaxial$ and $\etab > 0$ such that for any $Y \in \neigh$ and $\eta < \etab$, we have
	\begin{align}
		\label{eq:potential-decrease}
		\Phi(Y'') \le \inparen{1 - \eta K} \Phi(Y).
	\end{align}
	Here, $Y'' \in \MMC_{n / 2}$ is the point reached by a single step of Riemannian gradient descent starting from $Y$ with step size $\eta$. (The notation $Y''$ is used as $Y'$ is reserved for something else in the proof.) $K > 0$ is a constant which depends only on the instance of \eqref{eq:MC-BM}.
\end{lemma}

% For some reason I had to use the following link to get this figure to not appear on its own page: https://tex.stackexchange.com/questions/69869/image-taking-up-full-page
\begin{figure}[!ht]
	\centering
	\begin{tabular}{cl}\toprule
		Notation                                     & Description                                                                                            \\
		\midrule
		\(p\)                                        & \(p = n/2\) always for this section and may be used as a shorthand                                     \\ \als
		$A \in \sym^{n \times n}$                    & cost matrix taking the form \eqref{eq:cost-mat-setup}                                                  \\ \als
		$M \in \sym^{p \times p}$                    & one pseudo-PD block of the cost matrix; see \eqref{eq:cost-mat-setup}                                  \\ \als
		$\Phi$                                       & potential; see Lemma \ref{lem:decrease-in-potential}                                                   \\ \als
		$\yaxial$                                    & the axial position as in Definition \ref{def:axial-position}; in matrix block form: $I_p$ over $- I_p$ \\ \als
		$Y \in \MMC_p$                               & $\yaxial + \Delta$ (an arbitrary point near $\yaxial$)                                                 \\ \als
		$\Delta \in \R^{n \times p}$                 & perturbation matrix used to define $Y$; see the line above                                             \\ \als
		$Y' \in \R^{n \times p}$                     & $Y - \eta \grad \obj(Y)$ (the point we get to with a gradient step from \(Y\))                         \\ \als
		\(Y'' \in \MMC_p\)                           & the retracted (row-normalized) $Y'$ (equivalently the result of taking a                               \\ &single step of Riemannian gradient descent from $Y$) \\ \als
		$Z_i$                                        & used to denote the $i$th row of $Z$ (taken as a column vector) for                                     \\ & a given matrix $Z$ \\ \als
		\(\iplus, \iminus\)                          & \(i\) and \(i + p\) resp. for \(i \in [p]\)                                                            \\ \als
		$\Epsilon, \Epsilon' \in \R^{p \times p}$    & defined via their rows: $\Epsilon_i = \yiplus + \yiminus = \Delta_\iplus + \Delta_\iminus$             \\
		                                             & and $\Epsilon_i' = \yiplus' + \yiminus'$ for all $i \in [p]$                                           \\ \als
		$\Epsilon_{* i \setminus ii} \in \R^{p - 1}$ & the $i$th column of $\Epsilon$ with its $i$th entry removed (only used once!)                          \\ \als
		$M[i] \in \sym^{(p - 1) \times (p - 1)}$     & the submatrix of $M$ formed by removing its $i$th row and column                                       \\ \als
		$\lambda_{\min}(\cdot)$                      & minimum eigenvalue of the input                                                                        \\ \als
		$\minalleigs$                                & $\min_{\ell \in [p]} \lambda_{\min}(M[\ell])$                                                          \\ \als
		\(e_i\)                                      & the \(i\)th standard basis vector                                                                      \\ \als
		$|| \cdot ||$                                & Euclidean (or equivalently Frobenius) norm                                                             \\ \als
		\bottomrule
	\end{tabular}
	\caption{Notation guide for the proof of Lemma \ref{lem:decrease-in-potential}}
	\label{fig:local-min-notation}
\end{figure}

\begin{proof}
	We first provide a brief overview of the proof and introduce some notation. To begin, we will represent $Y$ explicitly in the form $Y = \yaxial + \Delta \in \MMC_{n / 2}$, where $\Delta \in \R^{n \times p}$ should be thought of as a perturbation matrix. (Recall that $p$ is always $n / 2$ in this section and may be used as a shorthand.) Using this representation, we derive explicit expressions for $Y'$ and then $\Phi(Y')$, where $Y' := Y - \eta \grad \obj(Y)$. Recalling the contents of Section \ref{subsec:Riemannian gradient descent}, $Y''$ is $Y'$ followed by a normalization of each row. So $Y'$ takes a gradient step from $Y$ but doesn't normalize the rows, meaning (assuming $\grad \obj(Y) \ne 0$) $Y' \notin \MMC_{n / 2}$. (Thus, we abuse notation here and extend the domain of $\Phi$ to $\R^{n \times p}$.) That said, it is easy to show $\Phi(Y'') \le \Phi(Y')$, so bounding $\Phi(Y')$ is sufficient. We are then able to bound $\Phi(Y')$ by the right-hand side of \eqref{eq:potential-decrease} by taking the step size and $\norm{\Delta}$ to be sufficiently small and using the pseudo-PD property of $M$.

	We now delve into the technical details. We will unfortunately need to introduce a significant amount of notation as we go since we will be performing the above analysis in a row-wise manner. (Which is natural in some sense when we recall that $\MMC_p$ is a product manifold formed by taking the Cartesian product of $n$ unit spheres in $\R^p$. And for product manifolds, geometric entities such as the tangent space and Riemannian derivatives can be expressed as products or concatenations of entities over the constituent manifolds.) As an aid, Figure \ref{fig:local-min-notation} can be used as a reference for the notation used in this proof.

	Letting $\iplus, \iminus$ denote $i, i + p$ respectively for $i \in [p]$, we first derive expressions for $\yiplus', \yiminus' \in \R^p$ for all $i \in [p]$. Recall that $Y' = Y - \eta \grad \obj(Y) = Y - 2 \eta (A - \diag(\nu))Y$ with $\nu$ defined as in \eqref{eq:nu}. Then
	\begin{align*}
		\yiplus' & = \yiplus - 2 \eta \inparen{ \sum_{j = 1}^n A_{ij} Y_j - \sum_{j = 1}^n A_{ij} \inangle{\yiplus, Y_j} \yiplus}                                 \\
		         & = \yiplus - 2 \eta \inparen{ \sum_{j = 1}^p M_{ij} (\yjplus + \yjminus) - \sum_{j = 1}^p M_{ij} \inangle{\yiplus, \yjplus + \yjminus} \yiplus} \\
		         & = \yiplus - 2 \eta  \sum_{j = 1}^p M_{ij} \inparen{\Epsilon_j - \inangle{\yiplus, \Epsilon_j} \yiplus}.
	\end{align*}
	The second line uses the block form of $A$, and the third line introduces new notation: we let $\Epsilon \in \R^{p \times p}$ be defined such that the $i$th row of $\Epsilon$ is $\Epsilon_i = \yiplus + \yiminus = \Delta_\iplus + \Delta_\iminus$ for $i \in [p]$. The matrix $\Epsilon$ is directly related to the potential $\Phi$; indeed, $\Phi(Y) = ||\Epsilon||^2$.

	Similarly, one can derive
	\begin{align*}
		\yiminus' = \yiminus - 2 \eta  \sum_{j = 1}^p M_{ij} \inparen{\Epsilon_j - \inangle{\yiminus, \Epsilon_j} \yiminus}.
	\end{align*}
	Next, we define $\Epsilon' \in \R^{p \times p}$ analogously to $\Epsilon$ but using $Y'$: the $i$th row of $\Epsilon'$ is $\Epsilon_i' = \yiplus' + \yiminus'$ for $i \in [p]$. Thus, $\Phi(Y') = ||\Epsilon'||^2$. (We abuse notation and extend the domain of $\Phi$ to $\R^{n \times p}$.) We will bound $\Phi(Y')$ through $\Epsilon'$ in a row-wise manner. Using the expressions we have derived, we have for $i \in [p]$:
	\begin{align*}
		\Epsilon_i' & = \yiplus' + \yiminus'                                                                                                                                 \\
		            & = \Epsilon_i + 2 \eta \sum_{j = 1}^p M_{ij} \insquare{\inangle{\yiplus, \Epsilon_j} \yiplus + \inangle{\yiminus, \Epsilon_j} \yiminus - 2 \Epsilon_j}.
	\end{align*}
	Then for $i \in [p]$,
	\begingroup
	\allowdisplaybreaks
	\begin{align*}
		\norm{\Epsilon_i'}^2
		 & = \sum_{\ell = 1}^p \inangle{\Epsilon_i', e_\ell}^2                                                                                                                                                                                                            \\
		 & = \sum_{\ell = 1}^p \insquare{ \inangle{\Epsilon_i, e_\ell} + 2 \eta \sum_{j = 1}^p M_{ij} \insquare{\inangle{\yiplus, \Epsilon_j} \inangle{\yiplus, e_\ell} + \inangle{\yiminus, \Epsilon_j} \inangle{\yiminus, e_\ell} - 2 \inangle{\Epsilon_j, e_\ell}} }^2 \\
		 & = \sum_{\ell = 1}^p \left[
			\inangle{\Epsilon_i, e_\ell}^2
		+ O(\eta^2) \vphantom{\sum_{k = 1}^k} \right.                                                                                                                                                                                                                     \\ % Terribly hacky but it works.
		 & \left. \quad  + 2 \eta \sum_{j = 1}^p M_{ij} \insquare{\inangle{\Epsilon_i, e_\ell} \inparen{
					\inangle{\yiplus, \Epsilon_j} \inangle{\yiplus, e_\ell}
					+ \inangle{\yiminus, \Epsilon_j} \inangle{\yiminus, e_\ell} }
		- 2 \inangle{\Epsilon_i, e_\ell} \inangle{\Epsilon_j, e_\ell}} \right]                                                                                                                                                                                            \\
		 & = \norm{\Epsilon_i}^2 + O(\eta^2)                                                                                                                                                                                                                              \\
		 & \quad
		+ 2 \eta \sum_{\ell = 1}^p
		\sum_{j = 1}^p M_{ij} \insquare{\inangle{\Epsilon_i, e_\ell} \inparen{
				\inangle{\yiplus, \Epsilon_j} \inangle{\yiplus, e_\ell}
				+ \inangle{\yiminus, \Epsilon_j} \inangle{\yiminus, e_\ell} }
			- 2 \inangle{\Epsilon_i, e_\ell} \inangle{\Epsilon_j, e_\ell}}.
	\end{align*}
	\endgroup
	The $O(\eta^2)$ hides terms that depend on the perturbation $\Delta$, but this will not matter as $\eta$ will be taken sufficiently small in the final step after a bound on $\norm{\Delta}$ is set. (The $O(\eta^2)$ also hides terms that depend on the instance of \eqref{eq:MC-BM}, but these do not matter for our purposes.)

	Then
	\begingroup
	\allowdisplaybreaks
	\begin{align}
		\label{eq:potential-of-Y'}
		\begin{split}
			\Phi(Y') &= \norm{\Epsilon'}^2 \\
			&= \sum_{i = 1}^p \norm{\Epsilon_i'}^2 \\
			&= \norm{\Epsilon}^2 + O(\eta^2) \\
			&\quad + 2 \eta \underbrace{  \sum_{i = 1}^p \sum_{\ell = 1}^p
				\sum_{j = 1}^p M_{ij}
				\insquare{\inangle{\Epsilon_i, e_\ell} \inparen{
						\inangle{\yiplus, \Epsilon_j} \inangle{\yiplus, e_\ell}
						+ \inangle{\yiminus, \Epsilon_j} \inangle{\yiminus, e_\ell} }}}_{\circled{1}} \\
			& \quad - 4 \eta \underbrace{  \sum_{i = 1}^p \sum_{\ell = 1}^p
				\sum_{j = 1}^p M_{ij}  \inangle{\Epsilon_i, e_\ell} \inangle{\Epsilon_j, e_\ell} }_{\circled{2}}.
		\end{split}
	\end{align}
	\endgroup

	We now upper bound \circled{1} and lower bound \circled{2} starting with the former, which relies on the key observation that when $i = \ell$, then $\inangle{\Epsilon_i, e_\ell} = \inangle{\Epsilon_i, e_i}$ is small, and when $i \ne \ell$, then $\inangle{\yiplus, e_\ell}$ and $\inangle{\yiminus, e_\ell}$ are small. Formally,
	\begingroup
	\allowdisplaybreaks
	\begin{align}
		\circled{1} & = \sum_{i = 1}^p \sum_{\substack{\ell = 1                     \\ \ell \ne i}}^p
		\sum_{j = 1}^p M_{ij}
		\insquare{\inangle{\Epsilon_i, e_\ell} \inparen{
				\inangle{\yiplus, \Epsilon_j} \inangle{\yiplus, e_\ell}
		+ \inangle{\yiminus, \Epsilon_j} \inangle{\yiminus, e_\ell} }} \nonumber    \\
		            & \quad + \sum_{i = 1}^p
		\sum_{j = 1}^p M_{ij}
		\insquare{\inangle{\Epsilon_i, e_i} \inparen{
				\inangle{\yiplus, \Epsilon_j} \inangle{\yiplus, e_i}
		+ \inangle{\yiminus, \Epsilon_j} \inangle{\yiminus, e_i} }} \nonumber       \\
		            & \le
		\sum_{i = 1}^p \sum_{\substack{\ell = 1                                     \\ \ell \ne i}}^p
		\sum_{j = 1}^p M_{ij}
		\insquare{ \norm{\Epsilon_i} \inparen{
				\norm{\Epsilon_j} \norm{\Delta_\iplus}
		+ \norm{\Epsilon_j} \norm{\Delta_{\iminus}} }} \label{eq:bounding-1-1}      \\
		            & \quad + \sum_{i = 1}^p
		\sum_{j = 1}^p M_{ij}
		\insquare{ \norm{\Epsilon_i} \norm{\Delta_\iplus - \Delta_\iminus}
			\inparen{
				\norm{\Epsilon_j}
				+ \norm{\Epsilon_j}
		}} \label{eq:bounding-1-2}                                                  \\
		            & = O(\norm{\Epsilon}^2 \norm{\Delta}). \label{eq:bounding-1-3}
	\end{align}
	\endgroup

	\eqref{eq:bounding-1-1} uses Cauchy-Schwarz (recall that $\yiplus, \yiminus$ are unit vectors by definition) and the fact that $\inangle{\yiplus, e_\ell} = \inangle{\Delta_\iplus, e_\ell} \le \norm{\Delta_\iplus}$ since $\yiplus = e_i + \Delta_\iplus$ by definition and $i \ne \ell$. (And $\inangle{\yiminus, e_\ell}$ can be bounded similarly.) \eqref{eq:bounding-1-2} uses Cauchy-Schwarz as well as the following key bound:
	\begin{align}
		\label{eq:key-bound}
		\begin{split}
			\inangle{\Epsilon_i, e_i} &= \inangle{\Delta_\iplus, e_i} - \inangle{\Delta_\iminus, - e_i} \\
			&= \frac{- \norm{\Delta_\iplus}^2 + \norm{\Delta_\iminus}^2}{2} \\
			&\le \norm{\Delta_\iplus + \Delta_\iminus} \norm{\Delta_\iplus - \Delta_\iminus} \\
			&= \norm{\Epsilon_i} \norm{\Delta_\iplus - \Delta_\iminus},
		\end{split}
	\end{align}
	where we have used Lemmas \ref{lem:close-to-orthogonal} and \ref{lem:reverse-triangle-squares} from Appendix \ref{app:minor-claims}. This bound is critical; a less tight bound would not work because $\norm{\Epsilon}$ may be much smaller than $\norm{\Delta}$. The big $O$ notation in line \eqref{eq:bounding-1-3} hides terms which depend on the instance of \eqref{eq:MC-BM}, but these don't matter for our purposes.

	We now turn our focus to lower bounding \circled{2}, which is the only place where we use the fact that $M$ is pseudo-PD. We have
	\begingroup
	\allowdisplaybreaks
	\begin{align}
		\circled{2} & =  \sum_{\ell = 1}^p \sum_{\substack{i = 1                                                                                                                                \\ i \ne \ell}}^p
		\sum_{\substack{j = 1                                                                                                                                                                   \\ j \ne \ell}}^p M_{ij}  \inangle{\Epsilon_i, e_\ell} \inangle{\Epsilon_j, e_\ell}  \nonumber \\
		            & \quad + \sum_{\ell = 1}^p \sum_{j = 1}^p M_{\ell j} \inangle{\Epsilon_\ell, e_\ell} \inangle{\Epsilon_j, e_\ell}
		+ \sum_{\ell = 1}^p \sum_{i = 1}^p M_{i \ell}  \inangle{\Epsilon_i, e_\ell} \inangle{\Epsilon_\ell, e_\ell} - \sum_{\ell = 1}^p
		M_{\ell \ell}  \inangle{\Epsilon_\ell, e_\ell} \inangle{\Epsilon_\ell, e_\ell} \nonumber                                                                                                \\
		            & \ge - O(\norm{\Epsilon}^2 \norm{\Delta}) + \sum_{\ell = 1}^p \sum_{\substack{i = 1                                                                                        \\ i \ne \ell}}^p
		\sum_{\substack{j = 1                                                                                                                                                                   \\ j \ne \ell}}^p M_{ij} \Epsilon_{i \ell} \Epsilon_{j \ell} \label{eq:bounding-2-1} \\
		            & = - O(\norm{\Epsilon}^2 \norm{\Delta}) + \sum_{\ell = 1}^p \Epsilon_{*\ell \setminus \ell \ell}^\top M[\ell] \Epsilon_{*\ell \setminus \ell \ell} \label{eq:bounding-2-2} \\
		            & \ge - O(\norm{\Epsilon}^2 \norm{\Delta}) + \sum_{\ell = 1}^p \minalleigs
		\norm{\Epsilon_{*\ell \setminus \ell \ell}}^2 \label{eq:bounding-2-3}                                                                                                                   \\
		            & = - O(\norm{\Epsilon}^2 \norm{\Delta}) + \minalleigs \inparen{\norm{\Epsilon}^2 - \sum_{i = 1}^p \Epsilon_{ii}^2} \label{eq:bounding-2-4}                                 \\
		            & = - O(\norm{\Epsilon}^2 \norm{\Delta}) + \minalleigs \inparen{\norm{\Epsilon}^2 - \sum_{i = 1}^p \inangle{\Epsilon_i, e_i}^2} \label{eq:bounding-2-5}                     \\
		            & \ge - O(\norm{\Epsilon}^2 \norm{\Delta}) + \minalleigs \inparen{\norm{\Epsilon}^2 - \sum_{i = 1}^p \norm{\Epsilon_i}^2 \norm{2 \Delta}^2} \label{eq:bounding-2-6}         \\
		            & = - O(\norm{\Epsilon}^2 \norm{\Delta}) + \minalleigs \inparen{1 - \norm{2 \Delta}^2} \norm{\Epsilon}^2 \label{eq:bounding-2-7}
	\end{align}
	\endgroup

	\eqref{eq:bounding-2-1} once again uses the key inequality \eqref{eq:key-bound} (and Cauchy-Schwarz) and also simply rewrote $\inangle{\Epsilon_i, e_\ell} = \Epsilon_{i \ell}, \inangle{\Epsilon_j, e_\ell} = \Epsilon_{j \ell}$. \eqref{eq:bounding-2-2} introduces the unfortunate notation $\Epsilon_{* \ell \setminus \ell \ell} \in \R^{p - 1}$, which denotes the $\ell$th column of $\Epsilon$ except the $\ell$th entry of this column (aka $\Epsilon_{\ell \ell}$) has been removed. Recall that $M[\ell] \in \sym^{(p - 1) \times (p - 1)}$ as always denotes the submatrix of $M$ formed by removing the $\ell$th row and column. Then, \eqref{eq:bounding-2-2} follows by observing that the inner two summations on the right side of \eqref{eq:bounding-2-1} form a quadratic form which is precisely $\Epsilon_{*\ell \setminus \ell \ell}^\top M[\ell] \Epsilon_{*\ell \setminus \ell \ell}$. In \eqref{eq:bounding-2-3}, we introduce the notation $\minalleigs$ which denotes $\min_{\ell \in [p]} \lambda_{\min}(M[\ell])$, where $\lambda_{\min}(\cdot)$ denotes the minimum eigenvalue of its argument. In other words, $\minalleigs$ lower bounds the eigenvalues of $M[\ell]$ for any $\ell$, and the fact that $\minalleigs > 0$ follows from the fact that $M$ is pseudo-PD. \eqref{eq:bounding-2-4} follows by expanding $\sum_{\ell = 1}^p \norm{\Epsilon_{* \ell \setminus \ell \ell}}^2$ and noting that only the diagonal entries of $\Epsilon$ are not covered. In \eqref{eq:bounding-2-5} we simply rewrite $\Epsilon_{ii} = \inangle{\Epsilon_i, e_i}$, and \eqref{eq:bounding-2-6} uses the key inequality \eqref{eq:key-bound}. \eqref{eq:bounding-2-7} simply uses $\sum_{i = 1}^p \norm{\Epsilon_i}^2 = \norm{\Epsilon}^2$.

	Now going back to \eqref{eq:potential-of-Y'} and using the bounds on \circled{1} and \circled{2}, we have
	\begin{align*}
		\Phi(Y') & \le \norm{\Epsilon}^2 - \eta \minalleigs \inparen{1 - \norm{2 \Delta}^2} \norm{\Epsilon}^2 + \eta O(\norm{\Epsilon}^2 \norm{\Delta}) + O(\eta^2) \\
		         & \le (1 - \eta K) \norm{\Epsilon}^2,
	\end{align*}
	where in the last line we took $||\Delta||$ and then $\eta$ to be sufficiently small. Finally, recall that $\Phi(Y) = \norm{\Epsilon}^2$, and note also that the norm of each row of $Y'$ is at least 1.\footnote{This follows because $Y'$ takes the form of ``a point on $\MMC_{n / 2}$ plus a tangent vector,'' and because tangent vectors are row-wise orthogonal to their base, clearly adding one can only increase the norm of each row.} Then Lemma \ref{lem:norm-potential} from Appendix \ref{app:minor-claims} implies $\Phi(Y'') \le \Phi(Y')$, and we are done. 
% 	(The intuition behind the latter is clear; normalizing each row should only make rows $i$ and $i + p$ more similar for all $i \in [p]$, thereby only decreasing the potential.) % It won't make them more similar; it will make them be more directly opposite each other. This is kind of awkward to say though.
\end{proof}

Next, we would like to extend the result of Lemma \ref{lem:decrease-in-potential} to all consecutive pairs of iterates produced by Riemannian gradient descent and not just the first pair. This will be done shortly in Lemma \ref{lem:Riemannian-grad-descent-stays-close}, but in preparation we first show that as long as the iterates of Riemannian gradient descent stay in the neighborhood $\neigh$ identified in Lemma \ref{lem:decrease-in-potential}, they form a Cauchy sequence where the distances between consecutive iterates decrease geometrically.

\begin{lemma}[Iterates confined to neighborhood form Cauchy sequence]
	\label{lem:iterates-are-Cauchy}
	Let $\neigh, \etab, K$ denote the neighborhood, step-size bound, and constant identified in Lemma \ref{lem:decrease-in-potential}. Suppose Riemannian gradient descent with step size $\eta < \etab$ is initialized at some $Y^{(0)} \in \neigh$, and furthermore $Y^{(1)}, \dots, Y^{(t)} \in \neigh$. Then for any $k \in \inbraces{0, \dots, t}$, we have
	\begin{align*}
		\smallnorm{Y^{(k)} - Y^{(k + 1)}} \le 4 \eta \norm{M} \smallnorm{Y^{(0)} - \yaxial} (1 - \eta K)^{k / 2}.
	\end{align*}
\end{lemma}

Before starting the proof, note that we critically do not require $Y^{(t + 1)} \in O$. This will be important in the proof of Lemma \ref{lem:Riemannian-grad-descent-stays-close}.

\begin{proof}
	We have for $k \in \{ 0, \dots, t\}$:
	\begin{align*}
		\norm{Y^{(k)} - Y^{(k + 1)}}^2 & = \norm{Y^{(k)} - \ret_{Y^{(k)}} \inparen{- \eta \grad \obj(Y^{(k)})}}^2.
	\end{align*}
	Now, recall from Section \ref{subsec:Riemannian gradient descent} that $\ret_{Y^{(k)}} \inparen{- \eta \grad \obj(Y^{(k)})}$ is equal to $Y^{(k)} - \eta \grad \obj(Y^{(k)})$ followed by a normalization of each row. We claim
	\begin{align*}
		\norm{Y^{(k)} - \ret_{Y^{(k)}} \inparen{- \eta \grad \obj(Y^{(k)})}}^2 & \le
		\norm{Y^{(k)} - \inparen{Y^{(k)} - \eta \grad \obj(Y^{(k)})}}^2                                               \\
		                                                                       & = \norm{\eta \grad \obj(Y^{(k)})}^2.
	\end{align*}
	Indeed, this follows from applying Lemma \ref{lem:proj-contractive} from Appendix \ref{app:minor-claims} in a row-wise manner, where $u$ is chosen as the $i$th row of $Y^{(k)}$, $v$ is chosen as the $i$th row of $\ret_{Y^{(k)}} \inparen{- \eta \grad \obj(Y^{(k)})}$, $\alpha = 1$, and $\beta$ is set to the norm of the $i$th row of $Y^{(k)} - \eta \grad \obj(Y^{(k)})$. ($\beta \ge 1$ since $\eta \grad \obj(Y^{(k)}$ is row-wise orthogonal to $Y^{(k)}$ due to the former being a tangent vector.) Intuitively, we are just using the fact that retraction onto the sphere can only bring you closer to your starting point.

		Next, note that
		\begin{align*}
			\norm{\eta \grad \obj(Y^{(k)})}^2 \le \norm{\eta \nabla \obj(Y^{(k)})}^2
			= \norm{2\eta AY^{(k)}}^2,
		\end{align*}
		where $\nabla$ is used to denote the classical Euclidean gradient. This follows because, as mentioned in Section \ref{subsec:Riemannian-derivatives}, the Riemannian gradient is equal to the Euclidean gradient composed with an orthogonal projection, and an orthogonal projection can only decrease the norm. Now write $Y^{(k)} = \begin{bmatrix}
		G_1 \\
		G_2
	\end{bmatrix}$ where $G_1, G_2 \in \R^{\frac{n}{2} \times \frac{n}{2}}$. Using the form of $A$ given in \eqref{eq:cost-mat-setup}, it is easy to see that
		\begin{align*}
			\smallnorm{2\eta AY^{(k)}}^2 = \norm{4 \eta M (G_1 + G_2)}^2 \le 16 \eta^2 \norm{M}^2 \norm{G_1 + G_2}^2,
		\end{align*}
		where we used the fact that the Frobenius norm is submultiplicative. Then, note that by definition, $\Phi(Y^{(k)}) = \norm{G_1 + G_2}^2$. Connecting the dots, we have shown up to this point that
		\begin{align}
			\label{eq:distance-bound-in-progress}
			\smallnorm{Y^{(k)} - Y^{(k + 1)}}^2 \le 16 \eta^2 \norm{M}^2 \Phi(Y^{(k)}).
		\end{align}
		Using Lemma \ref{lem:decrease-in-potential} and the fact that the iterates up to step $k$ are in $\neigh$, we have $\Phi(Y^{(k)}) \le \inparen{1 - \eta K}^k \Phi(Y^{(0)})$. (Note that when $k = t$, we are critically relying on the fact that Lemma \ref{lem:decrease-in-potential} does not require $Y'' \in \neigh$.) Furthermore, it is easy to see that $\Phi$ always bounds the distance away from $\yaxial$; indeed, let $Z = \begin{bmatrix}
		Z_1 \\
		Z_2
	\end{bmatrix}$ with $Z_1, Z_2 \in \R^{\frac{n}{2} \times \frac{n}{2}}$ and note that
	\begin{align*}
		\Phi(Z) & = \norm{Z_1 + Z_2}^2                    \\
		        & = \norm{Z_1 + I + Z_2 - I}^2            \\
		        & \le \norm{Z_1 + I}^2 + \norm{Z_2 - I}^2 \\
		        & = \smallnorm{Z - \yaxial}^2.
	\end{align*}
	\eqref{eq:distance-bound-in-progress} and the contents of the previous paragraph imply the desired result.
\end{proof}

Now we finally extend Lemma \ref{lem:decrease-in-potential} to all consecutive pairs of iterates produced by Riemannian gradient descent.
This can of course be achieved by confining all iterates to the neighborhood identified in Lemma \ref{lem:decrease-in-potential}. The following lemma does this by initializing in an even smaller neighborhood of $\yaxial$.

\begin{lemma}[Riemannian gradient descent stays close to $\yaxial$]
	\label{lem:Riemannian-grad-descent-stays-close}
	Let $\neigh, \etab, K$ denote the neighborhood, step-size bound, and constant identified in Lemma \ref{lem:decrease-in-potential}. Then there exists a neighborhood $\smallneigh \subseteq \neigh$ of $\yaxial$ such that if Riemannian gradient descent with step size $\eta < \min \inbraces{\etab, \frac{1}{2K}}$ is initialized at any $Y^{(0)} \in \smallneigh$, all future iterates lie in $\neigh$.
\end{lemma}

\begin{proof}
	We proceed by induction on the iterate counter $t$, and will choose $\smallneigh$ in the inductive step so that the proof goes through. $Y^{(0)} \in \neigh$ since $\smallneigh \subseteq \neigh$. Now suppose $Y^{(0)}, \dots, Y^{(t)} \in \neigh$, and we will set $\smallneigh$ independently of $t$ so that $Y^{(t + 1)} \in \neigh$. We have
	\begin{align*}
		\smallnorm{\yaxial - Y^{(t + 1)}} & \le \smallnorm{\yaxial - Y^{(0)}} + \sum_{k = 0}^t \smallnorm{Y^{(k)} - Y^{(k + 1)}}                                               \\
		                                  & \le \smallnorm{\yaxial - Y^{(0)}} + 4 \eta \smallnorm{M} \smallnorm{\yaxial - Y^{(0)}} \sum_{k = 0}^t \inparen{1 - \eta K}^{k / 2} \\
		                                  & \le \smallnorm{\yaxial - Y^{(0)}} \inparen{1 + 4 \eta \norm{M} \sum_{k = 0}^\infty \inparen{1 - \eta K}^{k / 2}},
	\end{align*}
	where we used Lemma \ref{lem:iterates-are-Cauchy} (which critically doesn't require $Y^{(t + 1)} \in \neigh$).

	Clearly $\inparen{1 + 4 \eta \norm{M} \sum_{k = 0}^\infty \inparen{1 - \eta K}^{k / 2}}$ is bounded and doesn't depend on $t$, so $\smallnorm{\yaxial - Y^{(0)}}$ can be chosen sufficiently small so that $Y^{(t + 1)} \in O$.
\end{proof}

Finally, we give the proof of Lemma \ref{lem:conv-to-point-with-same-obj}:

\begin{proofof}{Lemma \ref{lem:conv-to-point-with-same-obj}}
	We choose $N$ to be the neighborhood $\smallneigh$ identified in Lemma \ref{lem:Riemannian-grad-descent-stays-close} and pick $\etaprime = \min \inbraces{\etab, \frac{1}{2K}}$ with $\etab, K$ defined as in Lemma \ref{lem:decrease-in-potential}. Then due to Lemma \ref{lem:Riemannian-grad-descent-stays-close}, all iterates of Riemannian gradient descent with any step size $\eta < \etaprime$ initialized at any $Y^{(0)} \in N$ lie in $\neigh$, the neighborhood identified in Lemma \ref{lem:decrease-in-potential}. As a result, Lemma \ref{lem:iterates-are-Cauchy} implies that the sequence of iterates $Y^{(0)}, Y^{(1)}, \dots$ forms a Cauchy sequence, and since $\R^{n \times p}$ is complete, the iterates converge to some $\yconv$. Clearly $\yconv \in \MMC_{n / 2}$ as all iterates of Riemannian gradient descent lie on $\MMC_{n / 2}$.

	Also, Lemma \ref{lem:decrease-in-potential} and the fact that all iterates lie in $\neigh$ implies that the sequence $\Phi(Y^{(0)}), \Phi(Y^{(1)}), \dots$ converges to 0. Since $\Phi$ is continuous, we have $\Phi(\yconv) = 0$. It is easy to see that for $Z \in \MMC_{n / 2}$, we have $\Phi(Z) = 0$ if and only if $Z$ is \textit{antipodal}, where as in Section \ref{sec:local-minimum}, we say a point is antipodal if it takes the form $\begin{bmatrix}
			G \\
			-G
		\end{bmatrix} \in \MMC_{n / 2}$ for some $G \in \R^{\frac{n}{2} \times \frac{n}{2}}$. Thus, $\yconv$ is antipodal. Note that $\yaxial$ is also antipodal, and it is easy to check that all antipodal points have the same objective value. (In particular, if the cost matrix takes the form \eqref{eq:cost-mat-setup} as we assume without loss of generality in this section, that objective value is 0.)
\end{proofof}

% \section{Proof of convergence to a point with the same objective value (Lemma \ref{lem:conv-to-point-with-same-obj})}
% \label{sec:conv-same-obj-value}
% \input{Main Paper/conv_same_obj_value}

\section{Constructions of strictly pseudo-PD matrices (Lemma \ref{prop:existence-of-Mbad})}
\label{sec:bad-matrices-constructions}
In this section we provide two proofs of Lemma \ref{prop:existence-of-Mbad}, which posits the existence of $k \times k$ strictly pseudo-PD matrices for $k \ge 2$. The first is a probabilistic construction and the second is deterministic. The former has the advantage of having nonnegative entries, which, combined with Lemma \ref{lem:Yaxial-local-min}, results in ``natural'' cost matrices that can arise as the adjacency matrix of a weighted graph. For the latter deterministic construction, we also characterize the optimal solutions of the associated instances of \eqref{eq:MC-SDP}, revealing they have a qualitatively different structure than $\yaxial \yaxial^\top$.

\subsection{Probabilistic construction with nonnegative entries}
\label{subsec:probabilistic-construction}

Our random construction uses a random matrix $U \in \R^{k \times (k-1)}$ with nonnegative entries such that for each $i \in [k]$, the submatrix $U^{(i)} \in \R^{(k-1) \times (k-1)}$ formed by removing the $i$th row of $U$ has non-negligible least-singular-value (and hence full rank). Each entry of $U$ is generated i.i.d. from a $N(\mu, 1)$ where $\mu = c_0 \sqrt{\log k}$ for a sufficiently large constant $c_0>0$. The final matrix is just $M= UU^\top - \varepsilon I$, where $\varepsilon = k^{-\Omega(1)}$ is set appropriately. By construction $UU^\top$ is a $k \times k$ matrix of rank $k - 1$; hence $M$ has exactly one negative eigenvalue. In what follows $\lambda_\ell(M)$ denotes the $\ell$th eigenvalue of $M$ (note that $M$ is symmetric in our case; hence all eigenvalues are real).

\begin{proposition}[Randomized construction for Lemma~\ref{prop:existence-of-Mbad}] \label{prop:rand-rect-singular-value}
	There exists absolute constants $c_0, c_1, c_2>0$ such that the following holds for a given $k \in \N$ with $k \ge 2$. Suppose $U \in \R^{k \times (k - 1)}$ is a random matrix where each entry is generated i.i.d. from $N(\mu, 1)$ for $\mu=c_0 \sqrt{\log k}$ and let $M = UU^\top - \varepsilon I$ with $\varepsilon \coloneqq c_1/k^{c_2}$.  Then, with probability at least $1-1/k^{7}$, $M$ is nonnegative and strictly pseudo-PD. In particular,
	\begin{enumerate}
		\item[(i)] for each $i \in [k]$, the submatrix $M[i] \in \sym^{(k-1) \times (k-1)}$ formed by removing the $i$th row and the $i$th column of $M$ satisfies $\lambda_{k - 1} (M[i]) \ge \frac{c_1}{k^{c_2}}$.
		\item[(ii)] every entry of $M$ is nonnegative.
		\item[(iii)] it is not positive semidefinite i.e., $\lambda_{k}(M) \le - \frac{c_1}{k^{c_2}}$.
	\end{enumerate}
\end{proposition}

%The proof uses the following lemma. 
% \begin{lemma} \label{lem:rand-rect-singular-value}
% There exists absolute constants $c_0, c_1, c_2>0$ such that the following holds for a given $k \in \N$ with $k \ge 1$. Suppose $U \in \R^{(k+1) \times k}$ is a random matrix where each entry is generated i.i.d. from $N(\mu, 1)$ for $\mu=c_0 \sqrt{\log k}$.  Then, with probability at least $1-1/(k+1)^{7}$ we have:
% \begin{enumerate}
%     \item[(i)] for each $i \in [k+1]$, the submatrix $U^{(i)} \in \R^{k \times k}$ formed by removing the $i$th row of $U$ satisfies $s_k (U^{(i)}) > c_1/(k+1)^{c_2}$, where $s_k(A)$ denotes the least singular value of $A$.
%     \item[(ii)] every entry of $U$ is nonnegative.
% \end{enumerate}
% % For sufficiently large $k > 0$, there exists a matrix $U \in \R^{(k + 1) \times k}$ with nonnegative entries such that the submatrix $U^{(i)} \in \R^{k \times k}$ formed by removing the $i$th row of $U$ is full rank, for all $i \in [k + 1]$. (TODO: Should be edited to include poly lower bound on singular value)
% \end{lemma}

We remark that the constant $7$ in the exponent of the failure probability is arbitrarily chosen. We can make this an arbitrarily large constant, and adjust constants $c_0, c_1, c_2>0$ appropriately.

The proof of the above lemma using the following claim about the least singular value of square matrices. We remark that much stronger bounds on the least singular values are known in random matrix theory. We state and include a proof (which follows somewhat standard arguments) of the following claim which is more tailored for our needs.
\begin{lemma}\label{lem:rand-singular-value}
	Fix any $k \in \N$ with $k \ge 1$.
	Let $A \in \R^{k \times k}$ be a random matrix, each entry of which is sampled $N(\mu, 1)$ for some $\mu$. Then there exists absolute constants $c_1, c_2, c_3>0$ such that with probability at least $1-1/(k+1)^{8}$, we have:
	\begin{enumerate}
		\item[(i)] $s_k (A)^2 > 2c_1/(k+1)^{c_2}$, where $s_k(A)$ denotes the least singular value of $A$.
		\item[(ii)] every entry of $A$ is in the interval $[\mu - c_3 \sqrt{\log k}, \mu+ c_3 \sqrt{\log k}]$.
	\end{enumerate}
\end{lemma}
\begin{proof}

	We will show separately that both the required properties hold with high probability, and do a union bound over the failure of these two events.
	%\anote{To be filled in.}
	Part (ii) of the claim just follows by applying standard Gaussian tail bounds for a fixed entry of the $k \times k$ matrix, and then using a union bound over all the $k^2$ entries.

	We now focus on part (i) of the claim. This follows a standard argument using anti-concentration bounds (or small ball probability). %\cite{RudelsonV}
	Let $a_1, a_2, \dots, a_k \in \R^k$ represent the columns of $A$.  The least singular value of $A$ can be lower-bounded using the {\em leave-one-out} distance $\ell(A)$ which is defined as
	\begin{align}
		\ell(A)=\min_i \text{dist}(a_i, V_{-i}),
	\end{align}
	where $V_{-i}=\text{span}\{a_j : j\neq i\}$ and $\text{dist}(x, V)= \min_{v \in V} \norm{x-v}_2 $ is the perpendicular $\ell_2$ distance between $x$ and the subspace $V$. The least singular value $s_k(A)$ is related to $\ell(A)$ by
	\begin{equation}
		\frac{\ell(A)}{\sqrt{k}}\le s_{k}(A) \le \ell(A). \label{eq:leave-one-out}
	\end{equation}
	We now lower-bound the leave-one-out distance $\ell(A)$.
	Fix a column $i \in [k]$. Let $u$ be any unit vector in the subspace orthogonal to $V_{-i}=\text{span}\big(\{a_j: j \in [k] \setminus \{i\} \}\big)$; note that such a direction exists since $\text{dim}(V_{-i}) \le k-1$. Moreover, since the column $a_i \sim N(\mu {\bf 1}_k,I)$ where ${\bf 1}_k=(1,1,\dots, 1)$ is the all-ones vector, we have $a_i^\top u \sim N(z, 1)$ where $z = \mu u^\top {\bf}_k$.  From the anti-concentration of Gaussian $N(z,1)$, we have for an absolute constant $c>0$,
	$$\Pr\big[|u^\top a_i| \le \delta \big] \le \sup_{t \in \R} \Pr_{g\sim N(0,1)}\big[ g \in [t-\delta, t+\delta] \big] \le c \delta.$$
	By picking $\delta = 1/(c(k+1)^{9})$, we have with probability at least $1-\tfrac{1}{(k+1)^{9}}$,
	\begin{align*}
		\text{dist}(a_i, V_{-i}) & \ge |u^\top x| \ge \frac{1}{c (k+1)^{9}}.
	\end{align*}
	By applying a union bound over all the columns $i \in [k]$, we have $\ell(A) \ge 1/(k+1)^9$ with probability at least $1-\tfrac{1}{(k+1)^8}$. By applying \eqref{eq:leave-one-out}, we see that part (i) of the lemma also holds.

\end{proof}

\begin{proof}[Proof of Proposition~\ref{prop:rand-rect-singular-value}]

	Fix $i \in [k]$. The matrix $M[i]$ can be written in terms of the $(k-1) \times (k-1)$ submatrix $U^{(i)}$ as $M[i]=U^{(i)} (U^{(i)})^\top - \varepsilon I_{k-1}$. Each submatrix $U^{(i)}$ is a random matrix in $\R^{(k-1) \times (k-1)}$ with i.i.d. entries drawn from $N(\mu, 1)$ with $\mu > c_0 \sqrt{\log k}$. By applying  Lemma~\ref{lem:rand-singular-value} with $k-1$ and choosing $c_0 > 2c_3$, we get that with probability at least $1-1/k^8$ that $\sigma_{k-1}(U^{(i)})^2 \ge 2c_1/k^{c_2}$. Hence,
	$$ \sigma_{k-1}(M[i]) = \sigma_{k-1}\Big(U^{(i)} (U^{(i)})^\top \Big)- \varepsilon \ge \sigma_{k-1}(U^{(i)})^2 - \frac{c_1}{k^{c_2}} \ge \frac{c_1}{k^{c_2}}.$$
	Applying a union bound over all $i \in [k]$ proves the part (i) of the lemma.

	Part (ii) follows since $U$ (and hence $U U^\top$) has nonnegative entries and $I$ only has diagonal entries. So all the off-diagonal entries of $M$ are nonnegative. The non-negativity of the diagonal entries follows from the positive semi-definiteness of the $M[i]$.

	Finally part (iii) follows since $U U^\top$ is of rank $(k-1)$; hence $\lambda_k(M) =  - \varepsilon$, which gives the required bound.
\end{proof}

\subsection{Deterministic construction}
\label{subsec:deterministic-construction}

We provide another construction of $k \times k$ strictly pseudo-PD matrices for any \(k \ge 2\).  Unlike the construction given in Appendix \ref{subsec:probabilistic-construction}, this construction is fully deterministic.  However, it includes both positive and negative entries, making it an arguably less natural cost matrix for a weighted graph.

\begin{definition}[Almost-average matrix]
	\label{def:almost-average-mat}
	We define the \(k \times k\) almost-average matrix \(M\) as follows:
	\[M_{ij} = \begin{cases} 1 & \text{for } i = j,\\ -\frac{1}{k- 1.5} & \text{for } i \neq j. \end{cases}\]
\end{definition}

Now we show that \(M\) is strictly pseudo-PD:

\begin{proof}
	First we show that \(M\) is not positive semidefinite.  Consider a test vector \(x\) of all 1's:
	\begin{align*}
		(Mx)_i & = 1 - \frac{k - 1}{k - 1.5} = - \frac{0.5}{k - 1.5} \\
		Mx     & = -\frac{1}{2k - 3} \cdot x
	\end{align*}
	So \(x\) is an eigenvector with a negative eigenvalue.

	Now we show \(M[i] \succeq 0\) for any $i \in [k]$. (Recall that $M[i] \in \sym^{(k - 1) \times (k - 1)}$ denotes the submatrix of $M$ formed by deleting the $i$th row and column.)
	To see this, note that $M[i]$ is strictly diagonally dominant for any $i$ (for every row, the sum of the magnitudes of the non-diagonal entries is less than the magnitude of the diagonal entry).  Any symmetric, strictly diagonally dominant matrix with nonnegative diagonal entries is positive definite.  So \(M[i]\) is indeed positive definite.
\end{proof}

We additionally note that when the cost matrix takes the form \eqref{eq:Yaxial-local-min-obj} with $M$ defined as in Definition \ref{def:almost-average-mat} (setting $k = n / 2$), one can show that \eqref{eq:MC-SDP} has a unique optimal rank-one solution. This implies in particular that if we view this instance of \eqref{eq:MC-SDP} as a convex relaxation of a corresponding Max-Cut instance (albeit with perhaps unusual negative edge weights modeling ``attractions''), the relaxation is tight. The following proposition will be used in Section \ref{sec:experiments} (Experiments).

\begin{proposition}[Optimal solution for cost matrix arising from almost-average matrix]
	\label{prop:opt-almost-average}
	For any $n \ge 4$, consider the instance of \eqref{eq:MC-SDP} with cost matrix
	\begin{align}
		\label{eq:almost-avg-cost}
		A =
		\begin{bmatrix}
			M & M \\
			M & M
		\end{bmatrix}
	\end{align}
	where $M \in \sym^{\frac{n}{2} \times \frac{n}{2}}$ is defined as in Definition \ref{def:almost-average-mat} (setting $k = n / 2$). The unique optimal solution of such an instance is the matrix $J \in \sym^{n \times n}$ consisting of all $1$'s, implying in particular that the optimal objective value is
	\begin{align*}
		4p \inparen{1 - \frac{p - 1}{p - 1.5}},
	\end{align*}
	where $p$ is shorthand for $n / 2$.
\end{proposition}

Clearly one could extend Proposition \ref{prop:opt-almost-average} to the case where \eqref{eq:almost-avg-cost} is arbitrarily shifted by a diagonal matrix as in, e.g., Lemma \ref{lem:Yaxial-local-min}, but we state it this way for simplicity and because we will use precisely matrices of the form \eqref{eq:almost-avg-cost} in Section \ref{sec:experiments} (Experiments). Proposition \ref{prop:opt-almost-average} is interesting because it provides an example of an instance of \eqref{eq:MC-SDP} where $\yaxial \yaxial^\top$ and the unique globally optimal solution $J$ are qualitatively very different (e.g., one is rank $n / 2$ and the other rank 1).

Before giving the proof of Proposition \ref{prop:opt-almost-average}, we make the useful observation that \textit{any} strictly pseudo-PSD (and therefore also strictly pseudo-PD) matrix has at most one negative eigenvalue.\footnote{Whenever we use the phrases ``at most one negative eigenvalue,'' ``exactly one negative eigenvalue,'' etc., we are counting for multiplicity. So, for example, if a matrix has the negative eigenvalue -2 with multiplicity 4, this is counted as four separate negative eigenvalues (and, e.g., such a matrix could not be strictly pseudo-PSD due to Lemma \ref{lem:one-neg-eig}).}

\begin{lemma}
	\label{lem:one-neg-eig}
	Let $B \in \sym^{k \times k}$ be a strictly pseudo-PSD matrix for any $k \ge 2$. Then $B$ has exactly one negative eigenvalue.
\end{lemma}

\begin{proof}
	Any strictly pseudo-PSD matrix has at least one negative eigenvalue by definition. We show that if $B$ has more than one negative eigenvalue, one can construct an instance of \eqref{eq:MC-BM} which contradicts a theorem due to \cite{BVB18}. (It is the same theorem which ultimately yields the result that when $p > n / 2$, \eqref{eq:MC-BM} has no spurious second-order critical points.) Toward this goal, consider the instance of \eqref{eq:MC-BM} with cost matrix
	\begin{align*}
		A =
		\begin{bmatrix}
			B & B \\
			B & B
		\end{bmatrix} \in \sym^{n \times n},
	\end{align*}
	where we have defined $n := 2k$ for notational brevity. Due to Proposition \ref{lem:Yaxial-second-order-crit}, the axial position $\yaxial$ is a second-order critical point for this instance. Furthermore, it is easy to see that the multiplier $\nu \in \R^n$ associated with $\yaxial$ due to \eqref{eq:nu} is 0. Then Theorem 3.4 from \cite{BVB18} gives that $A - \diag(\nu) = A$ has at most
	\begin{align*}
		\floor*{
			\frac{1}{k} \inparen{ \dim \F_{\yaxial \yaxial^\top } - \frac{k(k + 1)}{2} + n }
		} =
		\floor*{
			\frac{1}{k} \inparen{ n - k }
		} = 1
	\end{align*}
	negative eigenvalue. Here, $\F_{\yaxial \yaxial^\top }$ denotes the face of the convex feasible region of \eqref{eq:MC-SDP} (also known as the elliptope) associated with $\yaxial \yaxial^\top$. (In other words, $\yaxial \yaxial^\top$ is in the relative interior of $\F_{\yaxial \yaxial^\top }$; see Definition 2.5 and Proposition 2.7 in \cite{BVB18}.) Above, we used the fact that $\dim \F_{\yaxial \yaxial^\top } = \frac{k(k + 1)}{2} - k$ due to \cite[Prop. 2.7]{BVB18}.

	Next, note that
	\begin{align*}
		A =
		\begin{bmatrix}
			B & B \\
			B & B
		\end{bmatrix} =
		\begin{bmatrix}
			1 & 1 \\
			1 & 1
		\end{bmatrix}
		\otimes B,
	\end{align*}
	where $\otimes$ denotes the Kronecker product. Since the eigenvalues of $\begin{bmatrix}
			1 & 1 \\
			1 & 1
		\end{bmatrix}$ are $\inbraces{0, 2}$, this decomposition implies that if $B$ had more than one negative eigenvalue, $A$ would have more than one negative eigenvalue, a contradiction.
\end{proof}

Now we prove the main claim:

\begin{proofof}{Proposition \ref{prop:opt-almost-average}}
	$p$ will be used as shorthand for $n / 2$ throughout this proof.
	Due to, e.g., \cite[Prop. 1]{WW20} or \cite[Prop. 2.8]{BVB18}, $J$ is optimal if there exists some $\beta \in \R^n$ such that, defining $S := A - \diag(\beta)$, we have $SJ = 0$ and $S \succeq 0$.\footnote{This comes from the fact that if these conditions are satisfied, the primal-dual pair $(J, \beta)$ satisfies the KKT conditions of \eqref{eq:MC-SDP}, implying optimality since \eqref{eq:MC-SDP} is a convex program.} We claim these conditions hold when each entry of $\beta$ is set to $b := 2\inparen{1 - \frac{p - 1}{p - 1.5}}$. Indeed, $SJ = 0$ can be observed directly. As for $S \succeq 0$, note that
	\begin{align*}
		S = \inparen{\begin{bmatrix}
				1 & 1 \\
				1 & 1
			\end{bmatrix} \otimes M} - \diag(\beta).
	\end{align*}
	Since the spectrum of $\begin{bmatrix}
			1 & 1 \\
			1 & 1
		\end{bmatrix}$ is $\inbraces{0, 2}$, the spectrum of $S$, denoted $\sigma(S)$, is precisely
	\begin{align}
		\label{eq:spectrum-of-S}
		\sigma(S) = \inbraces{- b} \cup \inbraces{2 \lambda - b : \lambda \in \sigma(M)}.
	\end{align}
	Clearly $-b \ge 0$, so all that is left is to check that all elements of $\inbraces{2 \lambda - b : \lambda \in \sigma(M)}$ are nonnegative. Due to Lemma \ref{lem:one-neg-eig}, $M$ has at most one negative eigenvalue. (Note that any strictly pseudo-PD matrix is strictly pseudo-PSD.) Furthermore, we found in the proof that $M$ is strictly pseudo-PD (directly below Definition \ref{def:almost-average-mat}) that the all 1's vector is an eigenvector of $M$ with eigenvalue $- \frac{0.5}{p - 1.5}$. Thus, this is the single negative eigenvalue of $M$. Finally, note that
	\begin{align*}
		- \frac{1}{p - 1.5} - b = - \frac{1}{p - 1.5} - 2\inparen{1 - \frac{p - 1}{p - 1.5}}
		= 0,
	\end{align*}
	so we conclude that $S \succeq 0$. Thus, we have established that $J$ is optimal.

	As for the uniqueness of $J$, it is a classical result that $J$ is an extreme point of the feasible region of \eqref{eq:MC-SDP} (aka the elliptope)---see for example Definition 2.6 and Proposition 2.7 in \cite{BVB18} or Appendix F.1 in \cite{WW20}. \cite[Prop. 2]{WW20} then gives that if strict complementary slackness holds, meaning $\rank(S) = n - \rank(J) = n - 1$, then $J$ is the unique optimal solution. Indeed, this follows due to \eqref{eq:spectrum-of-S}, the fact that $M$ only has a single negative eigenvalue, and the fact that $-b > 0$ for $p \ge 2$.
\end{proofof}

\section{Experiments}
\label{sec:experiments}
In this section, we empirically evaluate our construction of a spurious local minimum for \eqref{eq:MC-BM} in the setting where \(p = n/2\) (the largest possible value of \(p\) before we are guaranteed to have no spurious minima). Our experiments suggest that the spurious local minima we construct have surprisingly large basins of convergence. (In comparison, our theoretical results only guarantee the existence of \textit{some} positive measure basin of convergence.)

Our code and data can be found at \url{https://github.com/vaidehi8913/burer-monteiro}. This repository also contains a link to a visualizer for the $p = 2$ and $p = 3$ settings.

\paragraph{Instance generation.} For our experiments, we use the deterministic construction of pseudo-PD matrices given in Section \ref{subsec:deterministic-construction}. In other words, the cost matrix always takes the form
\begin{align*}
	A =
	\begin{bmatrix}
		M & M \\
		M & M
	\end{bmatrix}
\end{align*}
where $M \in \sym^{\frac{n}{2} \times \frac{n}{2}}$ is defined as in Definition \ref{def:almost-average-mat} (setting $k = n / 2$). We run trials for \(n = 4, 50, 200, 1000\), each time setting \(p = n/2\). Recall that due to Lemma \ref{lem:Yaxial-local-min}, the axial position $\yaxial$ is a spurious local minimum for such cost matrices. Furthermore, due to Proposition \ref{prop:opt-almost-average}, we know precisely what the optimal value is for such instances, allowing us to distinguish whether the optimization algorithm converged to a spurious point or a global optimum.

\paragraph{Optimization setup.}
We use the standard trust-region solver with default settings from the MATLAB manifold optimization package Manopt \cite{BMAS14}. It is a second-order method which uses the gradient of the objective function (which we provide) and an approximation of the Hessian of the objective function found using finite differences (this is done automatically by their implementation).

\begin{figure}[h]
	\centering
	\includegraphics[scale=0.5]{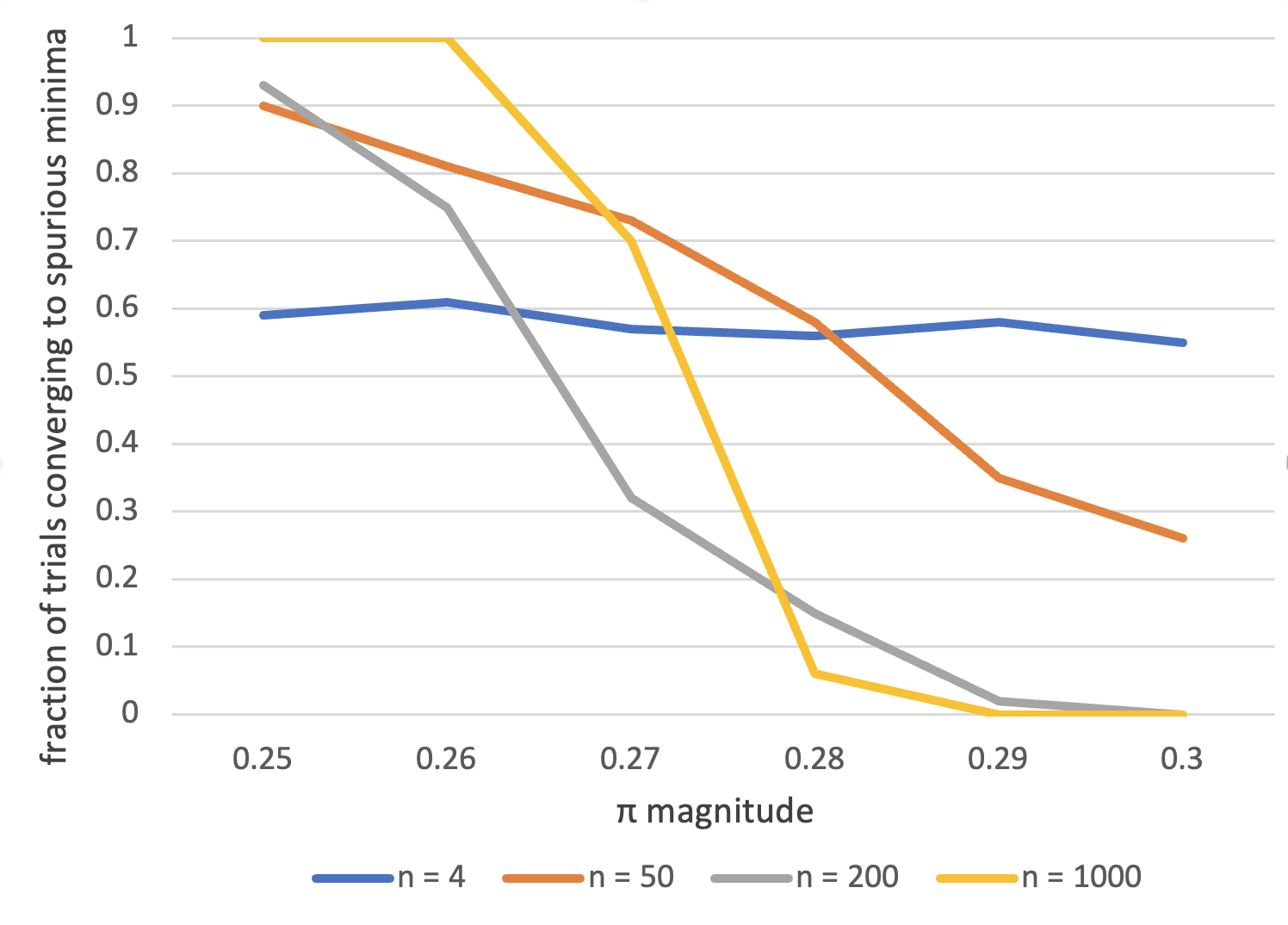}
	\caption{A plot summarizing our empirical results (data in Figure \ref{fig:deterministic-raw-data}). Note that there appears to be a phase transition at \(\pi \approx 0.27\), where trials almost always converge to spurious points for smaller \(\pi\) and almost always converge to global optima for larger \(\pi\).  This phase transition appears to grow sharper with larger values of \(n\).}
	\label{fig:deterministic-chart}
\end{figure}

\paragraph{Initialization.} In each trial, we sample the initialization point $Y^{(0)}$ from a neighborhood of the axial position $\yaxial$. We make use of the fact that $\M_p$ is a product manifold and measure the distance to initialization in a row-wise manner, making the following definition:

\begin{definition}[$\beta$-close]
	\label{def:beta-close}
	We say $Y, Y' \in \M_p$ are $\beta$-close if $\norm{Y_i - Y_i'} \le \beta$ for all $i \in [n]$, where $Y_i \in \R^p$ denotes the $i$th row of $Y$.
\end{definition}

For each trial, we choose a perturbation magnitude \(\pi > 0\), which specifies how far from \(\yaxial\) our initialization point will be. We then generate a perturbation matrix \(\Delta \in \mathbb{R}^{n \times p}\).  Each entry of \(\Delta\) is drawn \(\sim \mathcal{N}(0, \frac{1}{p})\) (a Gaussian distribution with variance \(\frac{1}{p}\)).  This ensures that \(\Delta\) has approximately unit-norm rows.  Then our initial point is given by
\[ Y^{(0)} = \rownorm \left(\yaxial + \pi \Delta \right) ,\]
where $\rownorm : \R^{n \times p} \to \R^{n \times p}$ normalizes each row of the input. This ensures that $Y^{(0)}$ is effectively sampled a constant distance away from $\yaxial$ on $\MMC_p$ in a uniformly chosen direction.

\paragraph{Data collection.} We report the fraction of trials in which the algorithm converged to a spurious point. We note that all trials we ran resulted in convergence to a point with objective value 0 (the objective value of our constructed spurious local minimum $\yaxial$) or to a point with the optimal (negative) objective value due to Proposition \ref{prop:opt-almost-average}.

\paragraph{Results.} We provide a summary of our experimental results in Figure \ref{fig:deterministic-chart} and the full data in Figure \ref{fig:deterministic-raw-data}.  For \(n = 4\), we ran 1000 trials for each reported value of \(\pi\).  For \(n = 50, 200\) we ran 100 trials for each \(\pi\), and for \(n = 1000\) we ran 50 trials for each reported value of \(\pi\).

We note an interesting phase transition that seems to occur at \(\pi \approx 0.27\). For perturbations greater than this threshold, the algorithm seems to almost always converge to a global optimum. Below this threshold, the algorithm seems to almost always converge to a spurious point. This threshold appears to get sharper as \(n\) gets larger. This suggests there is some \(\beta^* \approx 0.27\) such that points \(\beta^*\)-close to \(\yaxial\) are very likely to converge to a spurious point, and vice versa. (It is also interesting that this family of cost matrices seems to have this phase transition at the same value for every \(n\).) All in all, our experiments suggest a much larger basin of convergence for spurious minima than our theoretical results guarantee.

\begin{figure}[h]
	\centering
	\begin{tabular}{rccccc}\toprule
		$\pi$  & $n = 4$, & $50$, & $200$, & $1000$ \\
		\midrule
		$0.3$  & 0.55     & 0.26  & 0      & 0      \\
		$0.29$ & 0.58     & 0.35  & 0.02   & 0      \\
		$0.28$ & 0.56     & 0.58  & 0.15   & 0.06   \\
		$0.27$ & 0.57     & 0.73  & 0.32   & 0.70   \\
		$0.26$ & 0.61     & 0.81  & 0.75   & 1      \\
		$0.25$ & 0.59     & 0.90  & 0.93   & 1      \\ \bottomrule
	\end{tabular}

	\caption{ \textit{(Complete Data)} For each combination of \(n, \pi\), we report the fraction of our trials that converged to spurious points. For \(n = 4\), we ran 1000 trials for each reported value of \(\pi\).  For \(n = 50, 200\), we ran 100 trials for each reported value of \(\pi\).  For \(n = 1000\), we ran 50 trials for each reported value of \(\pi\). }
	\label{fig:deterministic-raw-data}
\end{figure}

\section{Conclusion}
\label{sec:conclusion}

We show that the Burer-Monteiro method can fail for instances of the Max-Cut SDP, for rank up to \(p = \frac{n}{2}\).  To the best of our knowledge, prior to our work it was unknown whether the Burer-Monteiro method could fail for any instance of any SDP with rank above the Barvinok-Pataki bound.
We settle this question and thus justify the use of smoothed analysis to obtain guarantees for the Burer-Monteiro method.

There are many interesting potential follow-up directions to this work.  We provide one construction of \eqref{eq:MC-BM} instances with spurious local minima.  Our construction has, and relies on, many interesting properties and symmetries.  It is possible that some of these properties are necessary, and further analyzing this construction could give insight that allows us to fully characterize spurious minima for \eqref{eq:MC-BM} instances.
Analyzing this construction could also help us understand this interesting threshold phenomenon for \eqref{eq:MC-BM} when $p=n/2$---one dimension higher and there are not only no spurious local minima, there are no spurious second-order critical points at all.
Another potential direction is seeing if similar techniques can be used to construct instances with spurious local minima for SDPs with other structures (not just Max-Cut).

Lastly, we note that a limitation of our work is that it only points to the existence of local minima, and does not give a full characterization of when we can expect local minima to exist.  We also note that since this is a theoretical result about optimization landscapes, we do not foresee any adverse societal impacts of our work.

% \subsubsection*{Acknowledgements}

% The authors thank the NeurIPS 2022 reviewers for many helpful comments.

	{
		% \small
		\bibliography{Other/references.bib}
	}

\newpage
\appendix

\section{Further discussion of prior work}
\label{app:prior-work}

The two most relevant papers to this work are \cite{BVB18} and \cite{WW20}. 
\cite{BVB18} excludes the presence of spurious second-order critical points for \eqref{eq:MC-BM} outside of a measure-zero set of cost matrices when $\frac{p (p + 1)}{2} > n$. (Furthermore, their result extends to a broad class of smooth SDPs.) Additionally, \cite{BVB18} shows that when $p > n / 2$, \eqref{eq:MC-BM} has no spurious second-order critical points. \cite{WW20} tightens the main lower bound of \cite{BVB18} to $\frac{p(p + 1)}{2} + p > n$ and also shows that when $\frac{p(p + 1)}{2} + p \le n$, there exists a set of cost matrices with non-zero measure whose corresponding instances of \eqref{eq:MC-BM} have spurious local minima. 

It was open to the best of our knowledge whether there exists any instance of \eqref{eq:MC-BM} with spurious second-order critical points when $\frac{p(p + 1)}{2} + p > n$. (In fact, this question was open for the broad class of smooth SDPs analyzed in \cite{BVB18}---see Section 6 in that paper.) We note that it is not clear how to extend the techniques of \cite{WW20} to the setting of our paper since their constructions critically rely on a technical assumption (the existence of ``minimally secant'' matrices) which provably never holds when $\frac{p(p + 1)}{2} > n$ (as they note in Appendix B). As a result, our paper takes a different approach.

There has also been a line of work \cite{BBJN18, PJB18, CM21} seeking to provide polynomial-time convergence guarantees to approximate global optima in a smoothed analysis setting. \cite{BBJN18} in particular performs smoothed analysis on an unconstrained quadratically-penalized version of \eqref{eq:MC-BM} (and its generalizations) and also provides a lower bound. However, their lower-bound construction does not apply to \eqref{eq:MC-BM} itself (or its generalizations). (In particular, their lower-bound construction sets the cost matrix to be 0, which does not work in our setting.)

Finally, \cite[Thm. 1]{MMMO17} implies that when the cost matrix $A$ comes from a weighted graph with nonnegative weights on the edges (as is typical in Max-Cut relaxations), any spurious local minimum still achieves the globally optimal value up to an $O(1 / p)$ multiplicative error. We note that as mentioned in Section \ref{sec:introduction}, there are many applications of \eqref{eq:MC-SDP} where $A$ may have negative entries and where the solution value is not as important as the optimal SDP solution itself (which is used to recover some ground-truth solution). As we showed in Proposition \ref{prop:opt-almost-average}, there exist instances of \eqref{eq:MC-SDP} where the optimal solution and the spurious point $\yaxial \yaxial^\top$ are qualitatively very different.

\section{Toolbox}
\label{app:toolbox}
\subsection{Proof of Proposition \ref{prop:char-first-order-opt}}

% NOTE: This is an older version of the proof of Proposition 4 which actually proves an analogous claim for the general Burer-Monteiro problem.

% \begin{proofof}{Theorem \ref{thm:BM-optimality}}
%     One direction is precisely Corollary 2.9 from \cite{BVB18}, that being that if $Y$ is a first-order critical point, then $C - \sum_{i = 1}^m \lambda_i A_i \succeq 0$ implies the optimality of $Y$ \eqref{eq:BM} and $YY^\top$ for \eqref{eq:SDP}.

%     Now let $Y$ be a globally optimal first-order critical point. Proposition 2.10 from \cite{BVB18} states that if strong duality holds for \eqref{eq:SDP}, then $C - \sum_{i = 1}^m \lambda A_i \succeq 0$. It was later shown in Appendix A.1 of \cite{WW20} that if LICQ holds over $\M_p$ for some $p \in \N_{> 0}$ such that $\M_p \ne \emptyset$, then strong duality holds for \eqref{eq:SDP}. Thus, LICQ holding over $\M_p$ removes the strong duality requirement of Proposition 2.10 from \cite{BVB18}, thereby completing the proof.
% \end{proofof}

Proposition \ref{prop:char-first-order-opt} follows directly from Corollary 2.9 and Proposition 2.10 from \cite{BVB18}, which give analogous claims for a more general class of programs. In regards to Corollary 2.9 (the ``if'' direction), note that \eqref{eq:MC-BM} trivially satisfies Assumption 1.1a (which is equivalent to the linear independence constraint qualification, aka LICQ, holding over the entire feasible region). Proposition 2.10 (the ``only if'' direction) requires strong duality, and strong duality holds for the convex program \eqref{eq:MC-SDP} since it satisfies Slater's condition.

\subsection{Riemannian gradient and Riemannian Hessian for (\ref{eq:MC-BM})} % For some reason Overleaf throws an error if I use \eqref instead of \ref here.

\begin{proposition}[Riemannian gradient for \eqref{eq:MC-BM}]
	\label{prop:Riemannian-grad}
	The Riemannian gradient of $\obj$ at $Y \in \MMC_p$ is given by
	\begin{align}
		\label{eq:Riemannian-grad-restated}
		\grad \obj(Y) = 2 (A - \diag(\nu)) Y, \quad \textup{where } \nu_i := \sum_{j = 1}^n A_{ij} \inangle{Y_i, Y_j}, \textup{ for all $i \in [n]$}.
	\end{align}
	Here, $Y_i \in \R^p$ denotes the $i$th row of $Y$, taken as a column vector.
\end{proposition}

\begin{proof}
	For a smooth objective function over a Riemannian submanifold of a vector space \cite[Def. 3.55]{Bou22}, the Riemannian gradient is given by the orthogonal projection of the Euclidean gradient to the tangent space \cite[Prop. 3.61]{Bou22}. Since $\M_p$ is a Riemannian submanifold of $\R^{n \times p}$ \cite[Sec. 2.1]{BVB18}, applying this yields
	\begin{align}
		\label{eq:Riemannian-grad-intermediate}
		\grad \obj(Y) = \Proj_Y \inparen{ 2 A Y} = 2 \Proj_Y (AY),
	\end{align}
	where the linear map $\Proj_Y : \R^{n \times p} \to \T_Y \MMC_p$ denotes the orthogonal projector onto $\T_Y \MMC_p \subseteq \R^{n \times p}$, i.e., $\Proj_Y (Z) = \argmin_{U \in \T_Y \MMC_p} \norm{U - Z}$. Since $\T_Y \MMC_p$ consists of those matrices in $\R^{n \times p}$ which are row-wise orthogonal to $Y$ (Proposition \ref{def:MC-BM-tangent-space}), it is clear that the orthogonal projection of $Z \in \R^{n \times p}$ onto $\T_Y \MMC_p$ is found by going row by row over $Z$ and each time deleting the component of row $i$ of $Z$ which lies in the span of row $i$ of $Y$. In other words,
	\begin{align}
		\label{eq:Proj_Y}
		\Proj_Y(Z) = Z - \diag(\mu) Y, \quad \text{where } \mu_i := \inangle{Z_i, Y_i}, \text{ for all $i \in [n]$}.
	\end{align}
	\eqref{eq:Riemannian-grad-intermediate} and \eqref{eq:Proj_Y} together yield our result.
\end{proof}

The Riemannian Hessian of $\obj$ at $Y \in \MMC_p$, $\Hess \obj(Y)$, is a linear, symmetric map from $\T_Y \MMC_p$ to $\T_Y \MMC_p$. For a Riemannian submanifold of a vector space such as $\MMC_p$, this is given by the classical differential of (a smooth extension of) $\grad \obj(Y)$, projected to the tangent space \cite[Cor. 5.16]{Bou22}.

We note that the Riemannian Hessian is the natural Riemannian analog of the Euclidean Hessian, and while the Euclidean Hessian $\nabla^2 f (x)$ of a function $f : \R^m \to \R$ at $x \in \R^m$ can be thought of as a symmetric $m \times m$ matrix containing the second-order partial derivatives of $f$ at $x$, it is often best understood as a linear map $\nabla^2 f(x) : \R^m \to \R^m$ defined via $\nabla^2 f(x)[u] = H u$, where $H \in \sym^{m \times m}$ is the aforementioned ``matrix form'' of $\nabla^2 f(x)$. With this viewpoint, $\nabla^2 f(x)[u]$ is the directional derivative of the Euclidean gradient $\nabla f(x)$ in the direction $u$. Similarly, while the Riemannian Hessian of $\obj$ at $Y \in \MMC_p$, $\Hess \obj(Y)$, could be identified with a symmetric $\dim (\MMC_p) \times \dim (\MMC_p)$ matrix,\footnote{Note that for a smooth manifold $\M$, $\dim(\M)$ is defined as $\dim(\T_x \M)$, where $\T_x \M$ denotes the tangent space (a vector space) at $x \in \M$. ($\dim(\T_x \M)$ is independent of $x$.)} it is best understood as a linear map $\Hess \obj(Y): \T_Y \MMC_p \to \T_Y \MMC_p$, where $\Hess \obj(Y)[U]$ denotes the ``directional derivative'' of the Riemannian gradient in the direction $U$. (The correct definition of ``directional derivative'' in this case is the Riemannian connection \cite[Thm. 5.6]{Bou22}.) As is standard, we take the latter form to be the definition of the Riemannian Hessian \cite[Def. 5.14]{Bou22}.

\begin{proposition}[Riemannian Hessian for \eqref{eq:MC-BM}]
	\label{prop:Riemannian-Hessian}
	The Riemannian Hessian of $\obj$ at $Y \in \MMC_p$, acting on $U \in \T_Y \MMC_p$, is given by
	\begin{align*}
		\Hess \obj(Y) [U] = 2 \Proj_Y \bigg( (A - \diag(\nu))U \bigg),
	\end{align*}
	where the linear map $\Proj_Y : \R^{n \times p} \to \T_Y \MMC_p$ denotes the orthogonal projector onto $\T_Y \MMC_p \subseteq \R^{n \times p}$, i.e., $\Proj_Y (Z) = \argmin_{U \in \T_Y \MMC_p} \norm{U - Z}$. $\nu$ is defined as in \eqref{eq:Riemannian-grad-restated}.
\end{proposition}

\begin{proof}
	This follows immediately from Equation 2.7 in \cite{BVB18}, which provides an expression for the Riemannian Hessian for a more general class of programs. (Their expression for the multiplier $\nu$, which they call $\mu$---see Equation 2.5 in that paper, is more complicated as it is for a more general class of programs, which is why we derived the Riemannian gradient from first principles in the proof of Proposition \ref{prop:Riemannian-grad}. However, the two expressions for this multiplier are ultimately equivalent for \eqref{eq:MC-BM} due to the uniqueness of the Riemannian gradient.) See also Section 7.7 of \cite{Bou22} for exposition (including an expression for the Riemannian Hessian) for a very general class of programs encompassing \eqref{eq:MC-BM}.
\end{proof}

\section{Extending construction of spurious local minimum to $p < n / 2$}
\label{app:extending-to-p<n/2}
Lemma \ref{lem:p-less-than-n/2} allows us to extend our construction of a spurious local minimum for the $p = n / 2$ case to $p < n / 2$. Indeed, it implies that our construction of a spurious local minimum for the instance of \eqref{eq:MC-BM} with associated feasible region $\MMC_{2p, p}$ yields a construction of a spurious local minimum for the instance of \eqref{eq:MC-BM} with associated feasible region $\MMC_{n', p}$, for all $n' \ge 2p$. Thus, one can construct an instance of \eqref{eq:MC-BM} with a spurious local minimum when $p < n / 2$ using the construction for the instance of \eqref{eq:MC-BM} associated with $\MMC_{2p, p}$.

We note that Lemma \ref{lem:p-less-than-n/2} also holds if you replace ``spurious local minimum'' with ``spurious first-order critical point'' or ``spurious second-order critical point,'' although we do not prove it here. However, the intuition is clear: we embed a ``bad instance'' into the higher-dimensional space, and design the cost matrix so that the bad instance does not interact with the added dimensions.

\begin{lemma}
	\label{lem:p-less-than-n/2}
	Let $(n, p) \in \N \times \N$ be such that there exists an instance of \eqref{eq:MC-BM} with cost matrix $A \in \sym^{n \times n}$ and feasible point $Y \in \MMC_{n, p}$ such that $Y$ is a spurious local minimum. Then for all $n' \ge n$, there exists a cost matrix $A' \in \sym^{n' \times n'}$ and a feasible point $Y' \in \MMC_{n', p}$ such that $Y'$ is a spurious local minimum for the instance of \eqref{eq:MC-BM} with cost matrix $A' \in \sym^{n' \times n'}$ and feasible region $\MMC_{n', p}$.
\end{lemma}

\begin{proof}
	We claim that the following construction works:
	\begin{align*} A' =
		\begin{bmatrix}
			A & 0 \\
			0 & 0
		\end{bmatrix} \in \sym^{n' \times n'},
		\qquad \qquad Y' =
		\begin{bmatrix}
			Y \\
			G
		\end{bmatrix} \in \MMC_{n', p},
	\end{align*}
	where $G \in \R^{(n' - n) \times p}$ is arbitrary except with the single restriction that all of its rows are unit vectors. Since $Y$ is a local minimum for the instance of \eqref{eq:MC-BM} with cost matrix $A$ and feasible region $\MMC_{n, p}$, there exists $\epsilon > 0$ such that for $Z \in \MMC_{n, p}$ with $\norm{Z - Y} < \epsilon$, we have $\inangle{A, YY^\top} \le \inangle{A, ZZ^\top}$. Now let $Z' \in \MMC_{n', p}$ with $||Z' - Y'|| < \epsilon$. Let $\bar{Z} \in \MMC_{n, p}$ denote the submatrix of $Z'$ consisting of the first $n$ rows. Then
	\begin{align*}
		\inangle{A', Y' Y'^\top } = \inangle{A, YY^\top} \le \inangle{A, \bar{Z} \bar{Z}^\top} = \inangle{A', Z' Z'^\top },
	\end{align*}
	since $||\bar{Z} - Y|| \le ||Z' - Y'|| < \epsilon$. Thus, $Y'$ is a local minimum.

	To see that $Y'$ is spurious, the spuriousness of $Y$ implies that there exists $V \in \MMC_{n, p}$ such that $\inangle{A, VV^\top} < \inangle{A, YY^\top}$. Then define $V' = \begin{bmatrix}
			V \\
			G
		\end{bmatrix} \in \MMC_{n', p}$, where $G \in \R^{(n' - n) \times p}$ is once again some matrix whose rows are unit vectors. Then,
	\begin{align*}
		\inangle{A', V' V'^\top} = \inangle{A, VV^\top} < \inangle{A, YY^\top} = \inangle{A', Y' Y'^\top}.
	\end{align*}
\end{proof}

\section{Further discussion of challenges for proving local minimality}
\label{app:challenges}
In this section, we further discuss challenges associated with proving $\yaxial$ is a local minimum in a ``traditional'' way. For example, \cite{WW20}, which constructs spurious local minima for \eqref{eq:MC-BM} when $p < \sqrt{2n}$, similarly first constructs spurious second-order critical points and then proves they are additionally local minima. However, their proof follows because their spurious second-order critical points are non-degenerate \cite[Def. 3 and Rem. 1]{WW20}, which corresponds to the rank of the Riemannian Hessian being sufficiently high. (See Remark 2 in that paper.) We show that \textit{every} spurious second-order critical point for \eqref{eq:MC-BM} is degenerate when $p$ is above the Barvinok-Pataki bound, meaning this approach won't work:

\begin{proposition}[Spurious second-order critical points are degenerate when $p \ge \sqrt{2n}$]
	\label{prop:degenerate-over-Barvinok-Pataki}
	Let $Y \in \MMC_p$ be a spurious second-order critical point for an (arbitrary) instance of \eqref{eq:MC-BM} where $\frac{p(p + 1)}{2} > n$. Then $Y$ is degenerate \textup{\cite[Def. 3 and Rem. 1]{WW20}}.
\end{proposition}

\begin{proof}
	Theorem 1.6 from \cite{BVB18} gives that any spurious second-order critical point for \eqref{eq:MC-BM} must be full rank, meaning $\rank(Y) = p$. Let $S := A - \diag(\nu)$, where $\nu$ is defined as in \eqref{eq:nu}. Then the first-order criticality of $Y$ (Proposition \ref{prop:MC-BM-first-order-crit}) implies $SY = 0$, meaning $\rank(S) \le n - p$. We have (see Section \ref{subsec:Riemannian-derivatives}):
	\begin{align}
		\inangle{\Hess \obj(Y)[U], U} & = 2 \inangle{S, UU^\top} \nonumber                                     \\
		                              & = 2 \vect(U)^\top \vect(S U I_p) \nonumber                             \\
		                              & = 2 \vect(U)^\top (I_p \otimes S) \vect(U), \label{eq:equiv-Hess-form}
	\end{align}
	for any $U \in \T_Y \MMC_p$. Here, $\vect(Z)$ stacks the columns of $Z$ on top of one another to convert $Z$ into a column vector, and $\otimes$ denotes the Kronecker product. In \eqref{eq:equiv-Hess-form}, we used the fact that $\vect(CXB) = (B^\top \otimes C) \vect(X)$.

	Then clearly $\rank(\Hess \obj(Y)) \le \rank(I_p \otimes S) = \rank(I_p) \cdot \rank(S) \le p(n - p)$. Recall from \cite[Rem. 1]{WW20} that $Y$ is degenerate if
	\begin{align*}
		\rank(\Hess \obj(Y)) < \dim(\M_p) - \frac{p(p - 1)}{2} = np - n - \frac{p(p - 1)}{2}.
	\end{align*}
	(See, e.g., \cite[Proposition 1.2]{BVB18} for the fact that $\dim(\M_p) = np - n$.) Then $Y$ is degenerate if
	\begin{align*}
		p(n - p) < np - n - \frac{p(p - 1)}{2} \iff \frac{p(p + 1)}{2} > n.
	\end{align*}
\end{proof}

\paragraph{Degeneracy of higher-order derivatives at $\yaxial$.} We know from Proposition \ref{lem:Yaxial-second-order-crit} that $\yaxial$ is a spurious second-order critical point if and only if the cost matrix takes the form
\begin{align}
	\label{eq:Yaxial-second-order-obj-app}
	A =
	\begin{bmatrix}
		P & P \\
		P & P
	\end{bmatrix} + \diag(\alpha)
\end{align}
for some $\alpha \in \R^n$ and strictly pseudo-PSD $P \in \sym^{\frac{n}{2} \times \frac{n}{2}}$. A natural question is whether higher-order Riemannian derivatives\footnote{See Section 10.7 of \cite{Bou22} for an introduction to higher-order Riemannian derivatives.} could be used to identify an additional condition under which $\yaxial$ is a (spurious) local minimum. For example, one may hope to show that when $P$ is additionally (strictly) pseudo-PD (the condition identified in Lemma \ref{lem:Yaxial-local-min}), then the fourth derivative is positive. Unfortunately this is not possible, as one can show that \textit{all} higher-order Riemannian derivatives at $\yaxial$ are degenerate when the cost matrix takes the form \eqref{eq:Yaxial-second-order-obj-app}.

This follows via an examination of $\T_\yaxial \MMC_{n / 2}$. Indeed, consider the subspace $W \subseteq \T_\yaxial \MMC_{n / 2}$ given by those matrices of the form
\begin{align}
	\label{eq:degenerate-subspace}
	W := \inbraces{
		\begin{bmatrix}
			G \\
			-G
		\end{bmatrix}
		: G \in \sym^{\frac{n}{2} \times \frac{n}{2}}, G_{ii} = 0 \text{ for all $i \in [n / 2]$}
	}.
\end{align}
It is easily observed that $W$ is contained in the kernel of $\Hess \obj (\yaxial)$ when the cost matrix takes the form \eqref{eq:Yaxial-second-order-obj-app}. Furthermore, one can show that $W$ is orthogonal to the vertical space at $\yaxial$ \cite[Def. 9.24]{Bou22} when we consider the quotient manifold $\MMC_{n / 2}^{\text{full}} / O(n / 2)$, where $\MMC_{n / 2}^{\text{full}}$ denotes the open subset of $\MMC_{n / 2}$ containing its rank $n / 2$ elements and $O(n / 2)$ denotes the orthogonal group in dimension $n / 2$. Indeed, the vertical space at $\yaxial$ consists of all tangent vectors of the form $\yaxial B$ where $B \in \R^{\frac{n}{2} \times \frac{n}{2}}$ is skew-symmetric. (See, e.g., p. 6 of \cite{WW20}.) Thus, the vertical space at $\yaxial$ is precisely the subspace of $\T_\yaxial \M_{n / 2}$ taking the form
\begin{align*}
	\inbraces{\begin{bmatrix}
			H \\
			-H
		\end{bmatrix} : H \in \R^{\frac{n}{2} \times \frac{n}{2}}, \text{$H$ is skew-symmetric}},
\end{align*}
which is clearly orthogonal to $W$.

Since tangent vectors in the subspace $W$ are in the kernel of $\Hess \obj (\yaxial)$ but not in the vertical space,\footnote{In fact, one can show that when $P$ in \eqref{eq:Yaxial-second-order-obj-app} is pseudo-PD, then $W$ contains \textit{all} tangent vectors which are in the kernel of $\Hess \obj (\yaxial)$ but not in the vertical space.} one may worry that following a smooth curve $c: I \to \MMC_{n / 2}$ ($I$ is an open interval of $\R$ containing 0) such that $c(0) = \yaxial, c'(0) \in W$ could yield a decrease in the objective value for sufficiently small inputs $t > 0$. Indeed, one must rule out such behavior to prove $\yaxial$ is a local minimum.

Unfortunately, higher-order Riemannian derivatives at $\yaxial$ all also contain $W$ in their zero sets, so they cannot a priori be used to rule out this behavior. Recall that formally, the $k$th Riemannian derivative of $\obj$ is a tensor field of order $k$ \cite[Def. 10.76]{Bou22} given by $\nabla^k \obj$, where $\nabla$ denotes the total covariant derivative \cite[Def. 10.77]{Bou22}. (Elsewhere in the paper we have used $\nabla$ to denote the classical Euclidean derivative, but the usage of $\nabla$ in this section is different; the total covariant derivative is \textit{not} (in general) the Euclidean derivative.) Furthermore, recall that tensor fields are pointwise objects, so we use the notation $\nabla^k \obj(Y) : \underbrace{\T_Y \M_{n / 2} \times \dots \times \T_Y \M_{n / 2}}_{\text{$k$ times}} \to \R$ to denote the $k$-linear function associated to a point $Y \in \MMC_{n / 2}$. Then we have the following result:

\begin{proposition}[Higher-order Riemannian derivatives at $\yaxial$ are degenerate]
	\label{lem:higher-order-degenerate}
	For \eqref{eq:MC-BM} when $p = n / 2$, suppose the cost matrix takes the form
	\begin{align}
		\label{eq:degenerate-cost-matrix}
		A =
		\begin{bmatrix}
			B & B \\
			B & B
		\end{bmatrix} + \diag(\alpha)
	\end{align}
	for some $\alpha \in \R^n$ and $B \in \sym^{\frac{n}{2} \times \frac{n}{2}}$.
	Then the subspace $W \subseteq \T_\yaxial \MMC_{n / 2}$ defined in \eqref{eq:degenerate-subspace} is in the zero set of $\nabla^k \obj(\yaxial)$ for all $k \ge 1$. (By this we mean $\nabla^k \obj(\yaxial)(U, \dots, U) = 0$ for any $U \in W$).
\end{proposition}

Note that Proposition \ref{lem:higher-order-degenerate} of course additionally encompasses the cases where $B$ is pseudo-PSD, pseudo-PD, etc.

\begin{proof}
	For notational brevity, let $p$ be shorthand for $n / 2$ in this proof.
	Let $c : I \to \MMC_p$ be a geodesic \cite[Def. 5.38]{Bou22} such that $c(0) = \yaxial, c'(0) = U \in W$. ($I$ is an open interval in $\R$ containing 0.) A simple extension of Example 10.81 in \cite{Bou22} implies
	\begin{align*}
		\nabla^k \obj(\yaxial)(U, \cdots, U) = (\obj \circ c)^{(k)} (0)
	\end{align*}
	for any $k \ge 1$. (Here, note that $\obj \circ c : I \to \R$, so $(\obj \circ c)^{(k)}$ is the ``usual'' $k$th derivative of a function from (an open interval in) $\R$ to $\R$.) Thus, it is sufficient to exhibit such a geodesic $c$ such that $\obj \circ c$ is constant over $I$ (implying $(\obj \circ c)^{(k)} (0) = 0$).

	To construct this geodesic, we will take advantage of the fact that $\MMC_p$ is a product manifold formed by the Cartesian product of $n$ unit spheres in $\R^p$. Recall that for the unit sphere $\sphere^{p - 1}$ with $x \in \sphere^{p - 1}, v \in \T_x \sphere^{p - 1}$, the curve
	\begin{align*}
		z_{x, v}(t) = \cos(t \norm{v}) x + \frac{\sin(t \norm{v})}{\norm{v}} v
	\end{align*}
	(with the usual smooth extension $\sin(t) / t = 0$ at $t = 0$) is a geodesic which traces the great circle on the sphere from $x$ in the direction $v$. (See Example 5.37 in \cite{Bou22}.) Of course, $z_{x, v} '(0) = v$.

	Then, viewing $\MMC_{p}$ in the form $(\underbrace{\sphere^{p - 1}, \dots, \sphere^{p - 1}}_{\text{$n$ times}})$ with the $i$th entry corresponding to the $i$th row in $\MMC_{p}$, we choose $c(t) = \inparen{ z_{\yaxial_1, U_1}(t), \dots, z_{\yaxial_n, U_n}(t) }$, where $\yaxial_i, U_i \in \R^p$ denote the $i$th rows of $\yaxial, U$ (taken as column vectors). Then $c(t)$ is a geodesic (e.g., \cite[Exerc. 5.39]{Bou22}) and $c(0) = \yaxial, c'(0) = U$. All that is left is to show that $\obj \circ c$ is constant. This follows due to the form of $W$; it is easy to check that for all sufficiently small $t$, we have that $c(t)$ is an \textit{antipodal configuration} as defined in Section \ref{sec:local-minimum}. (Recall that antipodal configurations take the form $\begin{bmatrix}
			G \\
			-G
		\end{bmatrix} \in \MMC_{n / 2}$ for some $G \in \R^{\frac{n}{2} \times \frac{n}{2}}$.) And it is easy to see that when the cost matrix takes the form \eqref{eq:degenerate-cost-matrix}, all antipodal configurations have the same objective value.
\end{proof}

Thus, we have identified (a subspace of) tangent vectors at $\yaxial$ which are not in the vertical space at $\yaxial$ and which lie in the zero sets of all higher-order Riemannian derivatives at $\yaxial$. As a result, it is not clear how higher-order Riemannian derivatives at $\yaxial$ can be used to prove $\yaxial$ is a local minimum.

\section{Proof that Riemannian gradient descent is nonincreasing (Lemma \ref{lem:obj-nonincreasing})}
\label{app:nonincreasing}
We first restate a result from \cite{Bou22} which yields a quadratic upper bound on the objective in the same style as the classic quadratic upper bound which holds in the Euclidean case when the Euclidean gradient is Lipschitz. (Indeed, the Riemannian gradient $\grad \obj$ for \eqref{eq:MC-BM} is Lipschitz, but defining Lipschitzness for the Riemannian gradient requires care \cite[Definition 10.44]{Bou22}. Proposition \ref{prop:quadratic-bound} is sufficient for our purposes since we only need a quadratic upper bound and don't care about the actual value of the Lipschitz constant.)

\begin{proposition}[Quadratic bound {\cite[Lem. 10.57, abbreviated]{Bou22}}]
	\label{prop:quadratic-bound}
	Consider a smooth manifold $\M$, retraction $\ret$ on $\M$ \textup{\cite[Definition 3.47]{Bou22}}, compact subset $\mathcal{K} \subseteq \M$, and continuous, nonnegative function $r: \compsub \to \R$. The set
	\begin{align*}
		\tangsub = \inbraces{ (x, s) \in \T \M : x \in \compsub \text{ and } \norm{s} \le r(x) }
	\end{align*}
	is compact in the tangent bundle $\T \M$ \textup{\cite[Definition 3.42]{Bou22}}. Assume $f : \M \to \R$ is twice continuously differentiable. Then there exists a constant $L$ such that, for all $(x, s) \in \tangsub$, we have
	\begin{align*}
		\left|
		f(\ret_x(s)) - f(x) - \inangle{s, \grad f(x)} \right| \le \frac{L}{2} \norm{s}^2.
	\end{align*}
\end{proposition}

We now give the proof of Lemma 5:

\begin{proofof}{Lemma \ref{lem:obj-nonincreasing}}
	Formally, our goal is to identify $\etatilde > 0$ such that for any $\eta < \etatilde$ and $Y \in \MMC_p$, we have
	\begin{align*}
		\obj (Y') \le \obj(Y) \quad \text{where $Y' = \ret_Y \inparen{- \eta \grad \obj(Y)}$},
	\end{align*}
	with $\ret_Y$ defined as the metric projection retraction for $\MMC_p$, as in Section \ref{subsec:Riemannian gradient descent}. (In fact, we will show that when $\grad \obj(Y) \ne 0$, our proof yields a strict decrease: $\obj (Y') < \obj(Y)$.)

	We apply Proposition \ref{prop:quadratic-bound} with $\compsub \gets \MMC_p$, as clearly $\MMC_p$ is compact. The fact that $\grad \obj$ is continuous implies $\norm{\grad \obj (\cdot)}$ is continuous, and this together with the compactness of $\MMC_p$ implies there exists some constant $P$ such that $\norm{\grad \obj(Z)} \le P$ for all $Z \in \MMC_p$.

	Pick $r : \MMC_p \to \R$ to be the constant function which sends everything to $P$, i.e., $r(Z) = P$ for all $Z \in \MMC_p$. Then for all $0 < \eta \le 1$ and $Z \in \MMC_p$, we have $(Z, - \eta \grad \obj(Z)) \in \tangsub$.

	Then Proposition \ref{prop:quadratic-bound} implies there exists some constant $L$ such that for $0 < \eta \le 1$, we have
	\begin{align*}
		\obj(Y') - \obj(Y) & = \obj \inparen{\ret_Y \inparen{- \eta \grad \obj(Y)}} - \obj(Y)                               \\
		                   & \le  \frac{L}{2} \norm{- \eta \grad \obj(Y)}^2 + \inangle{- \eta \grad \obj(Y), \grad \obj(x)} \\
		                   & = \inparen{ \frac{\eta^2 L }{2} - \eta} \norm{\grad \obj(Y)}^2
	\end{align*}
	Then $\frac{\eta^2 L }{2} - \eta < 0$ for all $0 < \eta < 2/L$, so setting $\etatilde \gets \min \inbraces{ 1, \frac{1}{L}}$ yields the desired result.
\end{proofof}

\section{Minor claims}
\label{app:minor-claims}
In this section, we prove minor claims that are used in Section \ref{sec:local-minimum}.

\begin{lemma}[Close to orthogonal]
	\label{lem:close-to-orthogonal}
	Let $v \in \R^m, \delta \in \R^m$ be such that $\norm{v} = \norm{v + \delta} = 1$. Then $\inangle{\delta, v} = - \norm{\delta}^2 / 2$.
\end{lemma}

\begin{proof}
	We have
	\begin{align*}
		         & \norm{v + \delta}^2 = 1                                  \\
		\implies & \norm{v}^2 + 2 \inangle{\delta, v} + \norm{\delta}^2 = 1 \\
		\implies & \inangle{\delta, v} = - \norm{\delta}^2 / 2.
	\end{align*}
\end{proof}

\begin{lemma}[Reverse triangle inequality with squares]
	\label{lem:reverse-triangle-squares}
	Let $u, v \in \R^m$. Then
	\begin{align*}
		\biggl \lvert \smallnorm{u}^2 - \smallnorm{v}^2  \biggr \rvert
		\le \norm{u + v} \norm{u - v}.
	\end{align*}
\end{lemma}

\begin{proof}
	Note that
	\begin{align*}
		\norm{u}^2 - \norm{v}^2 = \inangle{u + v, u - v} \le \norm{u + v} \norm{u - v}.
	\end{align*}
	The result follows by symmetry.
\end{proof}

\begin{lemma}[Normalizing doesn't increase the potential]
	\label{lem:norm-potential}
	Set $p = n / 2$, and let $\Phi$ be defined as in Lemma \ref{lem:decrease-in-potential}, although we abuse notation here and extend the domain to $\R^{n \times p}$. Let $Z \in \R^{n \times p}$ be arbitrary except with the single restriction that the norm of each of its rows is at least 1, and let $\zbar \in \MMC_{n / 2}$ denote the matrix formed by normalizing each row of $Z$. Then $\Phi(\zbar) \le \Phi(Z)$.
\end{lemma}

\begin{proof}
	It is clearly sufficient to show that $\norm{\zbar_i + \zbar_{i + p}}^2 \le \norm{Z_i + Z_{i + p}}^2$ for all $i \in [p]$. ($Z_i \in \R^p$ denotes the $i$th row.) This follows from Lemma \ref{lem:proj-contractive} below.
\end{proof}

\begin{lemma}[Metric projection is contractive]
	\label{lem:proj-contractive}
	Let $u, v \in \R^m$ be such that $\norm{u} = \norm{v} = 1$. Then for any $\alpha, \beta \ge 1$, we have
	\begin{align*}
		\norm{u - v}^2 \le \norm{\alpha u - \beta v}^2.
	\end{align*}
\end{lemma}

\begin{proof}
	We have
	\begin{align*}
		\norm{\alpha u - \beta v}^2 - \norm{u - v}^2
		 & = \alpha^2  + \beta^2 - 2 (\alpha \beta - 1) \inangle{u, v}  - 2 \\
		 & \ge \alpha^2 + \beta^2 - 2 (\alpha \beta - 1) - 2                \\
		 & = \alpha^2 + \beta^2 - 2 \alpha \beta                            \\
		 & = (\alpha - \beta)^2                                             \\
		 & \ge 0.
	\end{align*}
	In the second line, we used Cauchy-Schwarz and the fact that $\alpha \beta \ge 1$.
\end{proof}

\end{document}